\documentclass{article}

\usepackage[dvipsnames]{xcolor}

\usepackage{xspace}
\usepackage[ulem=normalem,draft]{changes}

\usepackage{amsthm,amsmath,amssymb}
\usepackage[bookmarks=true,bookmarksopen=true]{hyperref}

\usepackage{xspace}

\numberwithin{equation}{section}
\theoremstyle{plain}
\newtheorem{thrm}{Theorem}[section]
\newtheorem{lmm}[thrm]{Lemma}
\newtheorem{crllr}[thrm]{Corollary}
\newtheorem{prpstn}[thrm]{Proposition}
\newtheorem{dfntn}[thrm]{Definition}

\newtheorem{rmrk}[thrm]{Remark}

\numberwithin{equation}{section}

\renewcommand{\div}{\operatorname{div}}
\newcommand{\bY}{\mathbf{Y}}
\newcommand{\by}{\mathbf{y}}

\newcommand{\bz}{\mathbf{z}}
\newcommand{\bef}{\mathbf{f}}
\newcommand{\bg}{\mathbf{g}}
\newcommand{\be}{\mathbf{e}}

\newcommand{\p}{\mathfrak{p}}
\newcommand{\q}{\mathfrak{q}}

\newcommand{\bphi}{{\boldsymbol\phi}}
\newcommand{\bpsi}{{\boldsymbol\psi}}

\newcommand{\bL}{\mathbf{L}}

\newcommand{\bLd}{{\mathbf{L}^2(\Omega)}}
\newcommand{\bLqd}{{\mathbf{L}^2(Q)}}

\newcommand{\bLf}{{\mathbf{L}^4(\Omega)}}

\newcommand{\bLss}{{\mathbf{L}^s_\sigma(\Omega)}}
\newcommand{\bLscs}{{\mathbf{L}^{s'}_\sigma(\Omega)}}

\newcommand{\bLs}{{\mathbf{L}^{s}(\Omega)}}
\newcommand{\bWoT}{{\mathbf{W}(0,T)}}
\newcommand{\bWqpoT}{{\mathbf{W}_{q,p}(0,T)}}
\newcommand{\bWdpoT}{{\mathbf{W}_{2,p}(0,T)}}
\newcommand{\bWrsoT}{{\mathbf{W}_{r,s}(0,T)}}
\newcommand{\bV}{\mathbf{V}}

\newcommand{\bBrs}{{\mathbf{B}_{s,r}(\Omega)}}
\newcommand{\bBqp}{{\mathbf{B}_{p,q}(\Omega)}}
\newcommand{\bW}{{\mathbf{W}_{\!p}}}

\newcommand{\bWs}{{\mathbf{W}_s(\Omega)}}
\newcommand{\bWsc}{{\mathbf{W}_{s'}(\Omega)}}
\newcommand{\bWp}{{\mathbf{W}_p(\Omega)}}
\newcommand{\bWpc}{{\mathbf{W}_{\!p'}(\Omega)}}
\newcommand{\bWop}{{\mathbf{W}_0^{1,p}(\Omega)}}
\newcommand{\bWos}{{\mathbf{W}_0^{1,s}(\Omega)}}
\newcommand{\bWopc}{{\mathbf{W}_0^{1,p'}(\Omega)}}
\newcommand{\bWosc}{{\mathbf{W}_0^{1,s'}(\Omega)}}
\newcommand{\bWmop}{{\mathbf{W}^{-1,p}(\Omega)}}
\newcommand{\bWmos}{{\mathbf{W}^{-1,s}(\Omega)}}

\newcommand{\bH}{{\mathbf{H}}}
\newcommand{\bHo}{{\mathbf{H}_0^1(\Omega)}}
\newcommand{\bHmo}{{\mathbf{H}^{-1}(\Omega)}}

\newcommand{\Asq}{{A^{\frac{1}{2}}_s}}
\newcommand{\Asqc}{{A^{\frac{1}{2}}_{s'}}}

\newcommand{\bna}{\mathbf{\nabla}}
\newcommand{\mF}{\mathcal{F}}
\newcommand{\mA}{\mathcal A}

\newcommand{\mL}{\mathcal{L}}

\newcommand{\mY}{\mathcal{Y}}

\newcommand{\mG}{\mathcal{G}}

\newcommand{\mD}{\mathcal{D}}

\begin{document}

\title{Well-Posedness of Evolutionary Navier-Stokes Equations with Forces of Low Regularity on Two-Dimensional Domains\thanks{The first author was partially supported by Spanish Ministerio de Econom\'{\i}a y Competitividad under research project MTM2017-83185-P. The second was supported by the ERC advanced grant 668998 (OCLOC) under the EU’s H2020 research program.}}

\author{Eduardo Casas\thanks{Departmento de Matem\'{a}tica Aplicada y Ciencias de la Computaci\'{o}n, E.T.S.I. Industriales y de Telecomunicaci\'on, Universidad de Cantabria, 39005 Santander, Spain, {\tt eduardo.casas@unican.es}} \ and
Karl Kunisch\thanks{Institute for Mathematics and Scientific Computing, University of Graz, Heinrichstrasse 36, A-8010 Graz, Austria, {\tt karl.kunisch@uni-graz.at}}}

\maketitle

\begin{abstract}
\noindent Existence and uniqueness of solutions to the Navier-Stokes equations in dimension two with forces in the space $L^q( (0,T); \bWmop)$ for $p$ and $q$ in  appropriate parameter ranges are proven. The case of spatially measured-valued {forces} is included. For the associated Stokes equation the well-posedness results are verified in arbitrary dimensions {for any $1 < p, q < \infty$}.
\end{abstract}


\vspace{0.25cm}


\begin{center}
{\em Dedicated to Prof. Dr. Enrique Zuazua on the occasion of his 60th birthday}
\end{center}

\section{Introduction}
\label{S1}

In this paper we investigate the following Navier-Stokes system
\begin{equation}
\left\{\begin{array}{l}\displaystyle\frac{\partial \by}{\partial t} -\nu\Delta\by + (\by \cdot \bna)\by + \nabla\p = \bef\ \text{ in } Q = \Omega \times I,\\[1.2ex]\div\by = 0 \ \text{ in } Q, \ \by = 0 \ \text{ on } \Sigma = \Gamma \times I,\ \by(0) = \by_0 \text{ in } \Omega,\end{array}\right.
\label{E1.1}
\end{equation}
with focus on low regularity assumptions on the inhomogeneity $\bef$.
Here, $I = (0,T)$ with $0 < T < \infty$, and $\Omega\subset \mathbb{R}^d$ denotes a connected bounded domain  with a $C^3$ boundary $\Gamma$.

Our interest in this problem is two-fold. First, it has received very little attention in the literature so far. Indeed the only result which we are aware of is given in  \cite{Serre1983}, where $\bef$ is chosen in $W^{1,\infty}(I; \bWmop)$, with $\bWmop=\bigotimes_{i=1}^d W^{-1,p}(\Omega)$,  $d\in \{2,3\}$,  and $p\in (\frac{d}{2}, 2]$. It is mentioned there, that likely the result is not optimal, and the natural question arises whether, and how, it can be improved. Secondly we are interested in control problems with sparsity constraints, subject to \eqref{E1.1} as constraint. In this case it is natural to demand that for almost every $t\in I$ the forcing function is a vector valued Borel measure, i.e. $ \bef(t,\cdot) \in\mathbf{M}(\Omega) =\bigotimes_{i=1}^d M(\Omega) $, where $M(\Omega)$ is the space of real and regular Borel measures in $\Omega$, see e.g. \cite{Casas-Kunisch2019} and the references on sparse control given there. To treat \eqref{E1.1} with spatially  measure valued controls it is natural to consider the space $\bWmop$ with $p\in[1,\frac{d}{d-1})$, since in this case $\mathbf{M}(\Omega)\subset \bWmop$. In dimension $3$ this requires to consider the space $\bWmop$ with $p\in[1,\frac{3}{2})$. However, even in the stationary case the  existence of a solution  is an open issue for $d=3$ and this range of values for $p$, see \cite{Kim2009} or \cite{Serre1983}. For this reason we restrict our attention to the case $d=2$ throughout  the paper, unless specifically mentioned otherwise. For the two-dimensional case the result in \cite{Serre1983} guarantees the existence of a solution to  \eqref{E1.1} for $\bef \in W^{1,\infty}(I; \mathbf{M}(\Omega))$. But this regularity requirement with respect to time is not practical for control theory purposes.

Thus the focus of our work is the investigation of \eqref{E1.1} for $\bef \in L^q(I;\bWmop)$ in the case $\Omega \subset \mathbb{R}^2$, and $p<2$.
For this purpose we also require results on the Stokes equation associated to \eqref{E1.1}. Surprisingly, even this case has not yet been analysed  for $\bef \in L^q(I;\bWmop)$.  We carry out such an analysis which will be independent of the spatial dimension $d$.

Before we start, let us summarize, very selectively, some relevant literature.
In the stationary case the investigation of the Navier-Stokes system with data less regular than  $ \bLqd$ dates back to \cite{Leray1933}, who considers the case $\bef\in \mathbf{W}^{-1,2}(\Omega)$. In  \cite{Serre1983} the range of admissible forcing functions is increased to $\bef \in \bWmop$, with $p\in(\frac{d}{2},2)$
For more recent results we refer to \cite{FGS2006}, \cite{GSS2005},   \cite{Kim2009}, and the references there.

For the evolutionary  system well-posedness for forcing functions in the Hilbert spaces $L^2(I;\bLd)$,  and $L^2(I;\mathbf{W}^{-1,2}(\Omega))$, with $d\in\{2,3\}$, is well understood, see e.g. \cite{Boyer-Fabrie2013}, or \cite{Temam79}. The  analysis of the Stokes problem associated to  \eqref{E1.1} with forcing functions in the Bochner spaces $L^q(I;\mathbf{L}^p(\Omega))$, with $1<p,q < \infty,$ has attracted much attention. We refer to \cite{HS2018} for an informative summary, including the development of the maximal regularity techniques for this scenario.  Well-posedness of the Navier-Stokes system with  $\bef \in L^q(I;\mathbf{L}^p(\Omega))$
has been investigated in \cite{Giga-Sohr81} or \cite{Tolks18}. As mentioned above, the only work that we are aware of where  \eqref{E1.1} with forcing functions in Sobolev spaces with negative exponents has  been investigated is \cite{Serre1983}.

The plan of the paper is the following. In Section 2 the well-posedness results for the Stokes and the Navier-Stokes equations with $\bef \in L^q(I;\bWmop)$ under proper conditions on $p$ and $q$ are presented.  Some selected proofs are postponed to Sections 3 and 4. Section 5 presents a sensitivity analysis with respect to the right hand side, and Section 6 an asymptotic stability analysis. The proof of a  technical result on the nonlinearity appearing in \eqref{E1.1} is given in the Appendix.
\vspace{2mm}

\noindent\textbf{NOTATION}\vspace{2mm}

In this paper, we denote $\bLs = \bigotimes_{i=1}^dL^s(\Omega)$ and $\bWos = \bigotimes_{i=1}^dW_0^{1,s}(\Omega)$ for $s \in (1,\infty)$, and we choose the norm in $\bWos$ as
\[
\|\by\|_\bWos = \|\nabla\by\|_\bLs = \left(\int_\Omega|\nabla\by|^s\, dx\right)^{\frac{1}{s}}  = \left(\int_\Omega[\sum_{j = 1}^d|\nabla \by_j|^2 ]^{\frac{s}{2}}\, dx\right)^{\frac{1}{s}}.
\]
We also consider the spaces
\begin{align*}
&\bH = \text{closure of }\{\phi\in \mathbf{C}^\infty_0(\Omega):  \text{ div}\,  \phi =0 \}\text{ in }\bLd,\\
&\bWs = \{y \in \bWos : \div{y} = 0\}.
\end{align*}
For $s = 2$ we set $\bHo = \mathbf{W}_0^{1,2}(\Omega)$ and $\bV = \mathbf{W}_2(\Omega)$. We also define the following spaces
\begin{align*}
&\bWoT = \{\by \in L^2(I;\bV) : \frac{\partial \by}{\partial t} \in L^2(I;\bV')\}\\
&\bWrsoT = \{\by \in L^r(I;\bWs) : \frac{\partial \by}{\partial t} \in L^r(I;\bWsc')\}
\end{align*}
with $r, s \in (1,\infty)$, endowed with the norms
\begin{align*}
&\|\by\|_\bWoT = \|\by\|_{L^2(I;\bHo)} + \|\frac{\partial \by}{\partial t}\|_{L^2(I;\bV')},\\
&\|\by\|_\bWrsoT = \|\by\|_{L^r(I;\bWos)} + \|\frac{\partial \by}{\partial t}\|_{L^r(I;\bWsc')}.
\end{align*}
Obviously these are reflexive Banach spaces, and $\bWoT = \bWrsoT$ if $r = s = 2$.

Now we consider the interpolation space $\bBrs = (\bWsc',\bWs)_{1-1/r,r}$. From \cite[Chap. III/4.10.2]{Amann1995} we know that $\bWrsoT \subset C([0,T];\bBrs)$ and the trace mapping $\by \in \bWrsoT \to \by(0) \in \bBrs$ is surjective. If $r = s = 2$, then it is known that $\mathbf{B}_{2,2}(\Omega) = (\bV',\bV)_{\frac{1}{2},2} = \bH$. Hence, the embedding $\bWoT \subset C([0,T];\bH)$ holds; see \cite[Page 22, Proposition I-2.1]{Lions-Magenes68} and \cite[Page 143, Remark 3]{Triebel1978}.

\section{Well-Posedness Results}
\label{S2}
\setcounter{equation}{0}

The aim of this section is to prove the well-posedness of the following Navier-Stokes equations in dimension 2
\begin{equation}
\left\{\begin{array}{l}\displaystyle\frac{\partial \by}{\partial t} -\nu\Delta\by + (\by \cdot \bna)\by + \nabla\p = \bef\ \text{ in } Q,\\[1.2ex]\div\by = 0 \ \text{ in } Q, \ \by = 0 \ \text{ on } \Sigma,\ \by(0) = \by_0 \text{ in } \Omega,\end{array}\right.
\label{E2.1}
\end{equation}
where $\nu > 0$ is the kinematic viscosity coefficient, $\bef \in L^q(I;\bWmop)$, and $\by_0 \in \bY_0 = \bH + \bBqp$. The parameters $p$ and $q$ are fixed throughout this manuscript, and it is assumed  that
\begin{equation}
\frac{4}{3} \le p < 2 \text { and  }q > \frac{2p}{p - 1}
\label{E2.2}
\end{equation}
hold; with the exception of Corollary \ref{C2.1}. Observe that these assumptions imply that $q > 4$. {This condition \eqref{E2.2} is essential for the well-posedness of the bilinear form introduced in Lemma \ref{L2.1} as well as in the proofs of Theorem \ref{T2.1} and Proposition \ref{T2.3}. The low regularity of the force is due to the assumption $p < 2$. For $p \ge 2$ the solvability of \eqref{E2.1} is well known.}
The space $\bY_0$ is  endowed with the norm
\[
\|\by_0\|_{\bY_0} =  \inf_{\by_0 = \by_{01} + \by_{02}}\|\by_{01}\|_\bLd + \|\by_{02}\|_\bBqp,
\]
which makes it a Banach space. The assumption  $\Omega \subset \mathbb{R}^2$ is imposed throughout this paper,  except in Theorem \ref{T2.2}, which addresses the Stokes equation.

Now we introduce the following spaces:
\begin{align*}
&\bY = [L^2(I;\bV) \cap L^\infty(I;\bH)] + L^q(I;\bWp),\\
&\mY = \bWoT + \bWqpoT.
\end{align*}
They are Banach spaces with the norms
\begin{align*}
&\|\by\|_Y = \inf_{\by = \by_1 + \by_2}\|\by_1\|_{L^2(I;\bHo)} + \|\by_1\|_{L^\infty(I;\bLd)} + \|\by_2\|_{L^q(I;\bWop)},\\
&\|\by\|_\mY = \inf_{\by = \by_1 + \by_2}\|\by_1\|_\bWoT + \|\by_2\|_\bWqpoT.
\end{align*}
The solution of \eqref{E2.1} will be found in $\mY$. Before proving the existence of such a solution, let us present  the following technical lemma. Its proof is given in the Appendix.

\begin{lmm}
Assume that \eqref{E2.2} and $\Omega \subset \mathbb{R}^2$ hold. The bilinear operator $B:\bY \times \bY \longrightarrow L^2(I;\bHmo)$ defined by $B(\by_1,\by_2) = (\by_1 \cdot \nabla)\by_2$ is continuous.
\label{L2.1}
\end{lmm}

As usual, we can remove the pressure from the equation \eqref{E2.1} by using divergence free test functions.
\begin{dfntn}
We say that $\by \in \mY$ is a variational solution of \eqref{E2.1} if
\begin{equation}
\left\{
\begin{array}{l}
\displaystyle\langle\frac{d}{dt}\by(t),\bpsi\rangle_{\bWpc)' ,\bWpc} + a(\by(t),\bpsi) + b(\by(t),\by(t),\bpsi) \\
= \langle\bef(t),\bpsi\rangle_{\bWmop,\bWopc}\ \text{ in } (0,T),\ \ \forall \bpsi \in \bWpc,\\
\by(0) = \by_0,\end{array}\right.
\label{E2.3}
\end{equation}
where
\begin{align*}
&a(\by(t),\bpsi) = \nu\int_\Omega\nabla\by(x,t) : \nabla\bpsi(x)\, dx = \nu\sum_{i = 1}^2\int_\Omega\nabla\by_i(x,t)\nabla\bpsi_i(x)\, dx,\\
&b(\by(t),\by(t),\bpsi) = \langle B(\by(t),\by(t)),\bpsi \rangle_{\bHmo,\bHo} = \int_\Omega[\by(t) \cdot \nabla]\by(t)\cdot\nabla\bpsi\, dx.
\end{align*}
A distribution $\p$ in $Q$ is called an associated pressure if the equation
\[
\frac{\partial \by}{\partial t} -\nu\Delta\by + (\by \cdot \bna)\by + \nabla\p = \bef\ \text{ in } Q
\]
is satisfied in the distribution sense. Then, $(\by,\p)$ is called a solution of \eqref{E2.1}.
\label{D2.1}
\end{dfntn}

Given $\by$ satisfying \eqref{E2.3}, the pressure $\p$ is obtained by using De Rham's theorem; see \cite[Lemma IV-1.4.1]{Sohr2001}. The details are obtained in a similar way as in the proof for the  case of the Stokes equation, see step (iv) of the proof of Theorem \ref{T2.2} in Section \ref{S3}.

\begin{rmrk}
Given $s \in (1,\infty)$ and $\bg \in\bWmos$, we have that $\bg:\bWosc \longrightarrow \mathbb{R}$ is a linear and continuous mapping. We know that $\bWsc$ is a closed subspace of $\bWosc$. Therefore, for every element $\bg \in \bWmos$ we can consider its restriction to $\bWsc$, and $\|\bg\|_{\bWsc'} \le \|\bg\|_\bWmos$ holds. Moreover, from Hahn-Banach Theorem we know that every element of $\bWsc'$ is the restriction of an element of $\bWmos$. It is important to observe that the restriction of an element $\bg \in \bWmos$ to $\bWsc$ can be zero even though $\bg \neq 0$. Actually, given an element $\bg \in \bWsc'$, there are infinitely many elements in $\bWmos$ whose restriction to $\bWsc$ coincide with $\bg$. As a consequence, the variational solution $\by$ of \eqref{E2.1}, as defined by \eqref{E2.3}, only depends on the restriction of $\bef$ to $\bWpc$. Thus, different elements $\bef$ can lead to the same solution $\by$, but the pressure $\p$ changes. The pressure depends on the action of $\bef$ on the whole domain $\bWopc$.
\label{R2.1}
\end{rmrk}

As pointed out in Section 1, the embeddings $\bWoT \subset C([0,T];\bH)$ and $\bWqpoT \subset C([0,T];\bBqp)$ hold. Hence, $\mY \subset C([0,T];\bY_0)$ and, consequently, the initial condition $\by(0) = \by_0$ with $\by_0 \in \bY_0$ makes sense.

The next theorem is the main result of this section.

\begin{thrm}
Suppose that \eqref{E2.2} and $\Omega \subset \mathbb{R}^2$ hold. Then, system \eqref{E2.1} has a unique solution $(\by,\p) \in \mY \times W^{-1,q}(I;L^p(\Omega)/\mathbb{R})$. Furthermore, there exists a nondecreasing function $\eta_{p,q}:[0,\infty) \longrightarrow [0,\infty)$ with $\eta_{p,q}(0) = 0$ such that
\begin{equation}
\|\by\|_\mY \le \eta_{p,q}\Big(\|\bef\|_{L^q(I;\bWpc')} + \|\by_0\|_{\bY_0}\Big).
\label{E2.4}
\end{equation}
\label{T2.1}
\end{thrm}

For the proof of this result we will use the next two theorems. The first one concerns the associated Stokes equation and holds in arbitrary dimension $d$. It is given by
\begin{equation}
\left\{\begin{array}{l}\displaystyle\frac{\partial \by_S}{\partial t} -\nu\Delta\by_S +\nabla\p_S = \bg\ \text{ in } Q,\\[1.2ex]\div\by_S = 0 \ \text{ in } Q, \ \by_S = 0 \ \text{ on } \Sigma,\ \by_S(0) = \by_{S0} \text{ in } \Omega.\end{array}\right.
\label{E2.5}
\end{equation}
Given $\bg \in L^r(I;\bWmos)$ and $\by_{S0} \in \bBrs$ with $1 < r, s < \infty$, analogously to Definition \ref{D2.1}, we say that $\by_S \in \bWrsoT$ is a variational solution of \eqref{E2.5} if for every $\bpsi \in \bWsc$
\begin{equation}
\hspace{-0.25cm}\left\{\begin{array}{l}
\displaystyle\langle\frac{d}{dt}\by_S(t),\bpsi\rangle_{(\bWsc)' ,\bWsc} + a(\by_S(t),\bpsi) = \langle\bg(t),\bpsi\rangle_{\bWmos,\bWosc}\text{ in } (0,T),\\
\by_S(0) = \by_{S0}.\end{array}\right.
\label{E2.6}
\end{equation}
 A distribution $\p_S$ in $Q$ is called an associated pressure if the equation
\[
\frac{\partial \by_S}{\partial t} -\nu\Delta\by_S + \nabla\p_S = \bg\ \text{ in } Q
\]
is satisfied in the distribution sense.

\begin{thrm}
Assume that $\Omega \subset \mathbb{R}^d$ with $d \ge 2$. Given $\bg \in L^r(I;\bWmos)$ and $\by_{S0} \in \bBrs$ with $1 < r, s < \infty$, there exists a unique solution $(\by_S,\p_S) \in \bWrsoT \times W^{-1,r}(I;L^s(\Omega)/\mathbb{R})$ of \eqref{E2.5}.
Moreover, there exists a constant $C_{r,s}$ such that
\begin{equation}
\|\by_S\|_\bWrsoT \le C_{r,s}\Big(\|\bg\|_{L^r(I;\bWsc')} + \|\by_{S0}\|_\bBrs\Big).
\label{E2.7}
\end{equation}
\label{T2.2}
\end{thrm}

As mentioned at the end of section \ref{S1}, the embedding  $\bWrsoT \subset C([0,T];\bBrs)$ holds. Moreover, the trace mapping $\by \in \bWrsoT \to \by(0) \in \bBrs$ is continuous and surjective. This motivates our choice for the initial condition $\by_{S0} \in \bBrs$.

Though Theorem \ref{T2.2} is expected to hold by experts, it seems that there is no proof available in the literature. For this reason it is given in the next section. There, in Remark \ref{R3.1}, we shall also assert that Theorem \ref{T2.2} holds for domains which are only Lipschitz, provided that conditions on the ranges of $r$ and $s$ are met.

As a consequence of Theorem \ref{T2.2} we get the following corollary.
\begin{crllr}
Let assume that $p\in(2,\infty)$ and that $\Omega \subset \mathbb{R}^2$. Then, given $(\bef,\by_0) \in L^2(I;\bWmop) \times \mathbf{B}_{p,2}(\Omega)$, the system \eqref{E2.1} has a unique solution $(\by,\p) \in \bWoT \cap \mathbf{W}_{p',p}(0,T) \times W^{-1,p'}(I;L^2(\Omega)/\mathbb{R})$. Furthermore, there exist two constants $M_2 > 0$ and $M_p > 0$ such that
\begin{align}
\|\by\|_{\mathbf{W}_{p',p}(0,T)} &\le M_2\Big(\|\bef\|_{L^2(I;\bV')} + \|\by_0\|_\bH\Big)^2\notag\\
& + M_p\Big(\|\bef\|_{L^{p'}(I;\bWpc')} + \|\by_0\|_{\mathbf{B}_{p,2}(\Omega)}\Big).
\label{E2.8}
\end{align}
\label{C2.1}
\end{crllr}

\begin{proof}
From our assumptions on $p$, we have that $\bef \in L^2(I;\bWmop) \subset L^2(I;\bHmo)$ and $\by_0 \in \mathbf{B}_{p,2}(\Omega) \subset \mathbf{B}_{2,2}(\Omega) = \bH$. Hence, it is well known that \eqref{E2.1} has a unique solution $\by \in \bWoT$. Let us prove that $(\by\cdot\nabla)\by \in L^{p'}(I;\bWmop)$. Given an arbitrary element $\bpsi \in \bWopc$ and using \eqref{E7.1} with $r = 2p$ we infer
\begin{align*}
&|\langle(\by(t)\cdot\nabla)\by(t),\bpsi\rangle_{\bWmop,\bWopc}| = |b(\by(t),\bpsi,\by(t))|\\
& \le \|\by(t)\|^2_{\mathbf{L}^{2p}(\Omega)}\|\bpsi\|_\bWopc \le C_{2p}\|\by(t)\|^{\frac{2}{p}}_{\bLd}\|\by(t)\|^{\frac{2}{p'}}_\bHo\|\bpsi\|_\bWopc\\
&\le C_{2p}\|\by\|^{\frac{2}{p}}_{L^\infty(I;\bLd)}\|\by(t)\|^{\frac{2}{p'}}_\bHo\|\bpsi\|_\bWopc.
\end{align*}
Using this estimate and Young's inequality we deduce
\[
\|(\by\cdot\nabla)\by\|_{L^{p'}(I;\bWmop)} \le C_{2p}\|\by\|^{\frac{2}{p}}_{L^\infty(I;\bLd)}\|\by\|^{\frac{2}{p'}}_{L^2(I;\bHo)} \le C'\Big(\|\bef\|_{L^2(I;\bV')} + \|\by_0\|_\bH\Big)^2,
\]
where we have used the standard estimates for the solution of \eqref{E2.1} $\by \in L^2(I;\bHo) \cap L^\infty(I;\bH)$; see e.g. \cite[Theorem V.1.4]{Boyer-Fabrie2013}, or \cite[(3.135)]{Temam79}.

Since $\bef \in L^{p'}(I;\bWmop)$, we deduce from Theorem \ref{T2.2} with $\bg = \bef - (\by\cdot\nabla)\by$ that $\by \in \mathbf{W}_{p',p}(0,T)$ and
\begin{align*}
&\|\by\|_{\mathbf{W}_{p',p}(0,T)} \le C_{p',p}\Big(\|\bg\|_{L^{p'}(I;\bWpc')} + \|\by_0\|_{\mathbf{B}_{p,2}(\Omega)}\Big)\\
&\le C_{p',p}\Big[\|\bef\|_{L^{p'}(I;\bWpc')} + C'\Big(\|\bef\|_{L^2(I;\bV')} + \|\by_0\|_\bH\Big)^2 + \|\by_0\|_{\mathbf{B}_{p,2}(\Omega)}\Big],
\end{align*}
which implies \eqref{E2.8}.
\end{proof}

\begin{prpstn}
Suppose that \eqref{E2.2} and $\Omega \subset \mathbb{R}^2$ hold. Given $(\bg,\by_{N0}) \in L^2(I;\bHmo) \times \bH$, $\be_1, \be_2 \in \bY$, and $\nu_0 \ge 0$, then the system
\begin{equation}
\left\{\begin{array}{l}\displaystyle\frac{\partial \by_N}{\partial t} -\nu\Delta\by_N + \nu_0(\by_N \cdot \bna)\by_N + (\be_1 \cdot \bna)\by_N + (\by_N \cdot \bna)\be_2 + \nabla\p_N = \bg\ \text{ in } Q,\\[1.2ex]\div\by_N = 0 \ \text{ in } Q, \ \by_N = 0 \ \text{ on } \Sigma,\ \by_N(0) = \by_{N0} \text{ in } \Omega\end{array}\right.
\label{E2.9}
\end{equation}
has a unique solution $(\by_N,\p_N) \in \bWoT \times W^{-1,\infty}(I;L^2(\Omega)/\mathbb{R})$. Furthermore, there exists a nondecreasing function $\eta_N:[0,\infty) \longrightarrow (0,\infty)$ such that
\begin{equation}
\begin{array}{l}
\hspace{-0.2cm}\|\by_N\|_{L^2(I;\bHo)} + \|\by\|_{L^\infty(I;\bLd)}  \le \eta_N\Big(\|\be_2\|_\bY\Big)\Big(\|\bg\|_{L^2(I;\bV')} + \|\by_{N0}\|_{\bLd}\Big),\\
\hspace{-0.2cm}\|\by_N\|_\bWoT \le \nu_0\eta_N^2(\|\be_2\|_\bY)\Big(\|\bg\|_{L^2(I;\bV')} + \|\by_{N0}\|_{\bLd}\Big)^2\\
\hspace{0.6cm}+ [(1 +\nu +\|\be_1\|_\bY + \|\be_2\|_\bY)\eta_N(\|\be_2\|_\bY)+1]\Big(\|\bg\|_{L^2(I;\bV')} + \|\by_{N0}\|_{\bLd}\Big).
\end{array}
\label{E2.10}
\end{equation}
\label{T2.3}
\end{prpstn}
Similarly to the previous cases, we say that $\by_N \in \bWoT$ is a variational solution of \eqref{E2.9} if

\begin{equation}
\left\{
\begin{array}{l}
\displaystyle\frac{d}{dt}(\by_N(t),\bpsi)_\bLd + a(\by_N(t),\bpsi) + \nu_0b(\by_N(t),\by_N(t),\bpsi)\vspace{2mm}\\
+ b(\be_1(t),\by_N(t),\bpsi) + b(\by_N(t),\be_2(t),\bpsi) \\
= \langle\bg(t),\bpsi\rangle_{\bHmo,\bHo}\ \text{ in } (0,T),\ \ \forall \bpsi \in \bV,\\
\by_N(0) = \by_{N0}.\end{array}\right.
\label{E2.11}
\end{equation}
Furthermore, a distribution $\p_N$ in $Q$ is called an associated pressure if the equation
\[
\frac{\partial \by_N}{\partial t} -\nu\Delta\by_N + \nu_0(\by_N \cdot \bna)\by_N + (\be_1 \cdot \bna)\by_N + (\by_N \cdot \bna)\be_2 + \nabla\p_N = \bg\ \text{ in } Q
\]
is satisfied in the distribution sense.

Theorem \ref{T2.2} and Proposition \ref{T2.3} will be proved in Sections \ref{S3} and \ref{S4}, respectively.

\begin{proof}[Proof of Theorem \ref{T2.1}]

We are going to prove the existence of a solution $\by = \by_N + \by_S$ with $\by_N \in \bWoT$ and $\by_S \in \bWqpoT$. To this end we write $\by_0 = \by_{N0} + \by_{S0}$ with $\by_{N0} \in H$ and $\by_{S0} \in \bBqp$. Using Theorem \ref{T2.2} with $r=q$ and $s = p$, we define the function $\by_S \in \bWqpoT$ as the unique solution of the system
\begin{equation}
\left\{\begin{array}{l}\displaystyle\frac{\partial \by_S}{\partial t} -\nu\Delta\by_S +\nabla\p_S = \bef\ \text{ in } Q,\\[1.2ex]\div\by_S = 0 \ \text{ in } Q, \ \by_S = 0 \ \text{ on } \Sigma,\ \by_S(0) = \by_{S0} \text{ in } \Omega.\end{array}\right.
\label{E2.12}
\end{equation}
Now, we take $\by_N$ as the solution of
\begin{equation}
\left\{\begin{array}{l}\displaystyle\frac{\partial \by_N}{\partial t} -\nu\Delta\by_N + (\by_N \cdot \bna)\by_N + (\by_S \cdot \bna)\by_N + (\by_N \cdot \bna)\by_S + \nabla\p_N = -(\by_S\cdot\nabla)\by_S\ \text{ in } Q,\\[1.2ex]\div\by_N = 0 \ \text{ in } Q, \ \by_N = 0 \ \text{ on } \Sigma,\ \by_N(0) = \by_{N0} \text{ in } \Omega\end{array}\right.
\label{E2.13}
\end{equation}
The existence and uniqueness of the solution $\by_N \in \bWoT$ of the above system follows from Proposition \ref{T2.3} by taking $\nu_0 = 1$, $\be_1 = \be_2 = \by_S \in \bY$, and $\bg(t) = -B(\by_S(t),\by_S(t))$. As a consequence of Lemma \ref{L2.1} we have that $\bg \in L^2(I;\bHmo)$. Now, setting $\by = \by_N + \by_S$ and $\p = \p_N + \p_S$, and adding equations \eqref{E2.12} and \eqref{E2.13} we obtain that $\by \in \mY$, $\p \in W^{-1,q}(I;L^p(\Omega)/\mathbb{R}))$, and $(\by,\p)$ is a solution of \eqref{E2.1}. Moreover,  \eqref{E2.4} follows from \eqref{E7.6} to estimate $\bg$, \eqref{E2.7} and \eqref{E2.10}.

It remains to prove the uniqueness. Let $\by_1, \by_2 \in \mY$ and $\p_1, \p_2 \in W^{-1,q}(I;L^p(\Omega)/\mathbb{R})$ such that $(\by_1,\p_1)$ and $(\by_2,\p_2)$ are two solutions of \eqref{E2.1}. We take $(\by,\p) = (\by_2-\by_1,\p_2-\p_1)$. Subtracting the equations satisfied for both solutions we have
\begin{equation}
\left\{\begin{array}{l}\displaystyle\frac{\partial \by}{\partial t} -\nu\Delta\by + \nabla\p = -(\by_1 \cdot \nabla)\by - (\by\cdot\nabla)\by_2\ \text{ in } Q,\\[1.2ex]\div\by = 0 \ \text{ in } Q, \ \by = 0 \ \text{ on } \Sigma,\ \by(0) = 0 \text{ in } \Omega.\end{array}\right.
\label{E2.14}
\end{equation}

The right hand side of the above equation can be written in the form $\bg = -B(\by_1,\by) - B(\by,\by_2)$, which belongs to $L^2(I;\bHmo)$ by Lemma \ref{L2.1}. Let us prove that $\by \in \bWoT$. First, we observe that due to the properties of $p$ and $q$, in particular $p < 2$ and $q > 2$, we have that $\bV \subset \bWp$ and $\bWpc \subset \bV$. This implies that $\bV' \subset \bWpc'$ and, hence, $\bWoT \subset \bWdpoT$ and $\bWqpoT \subset \bWdpoT$. Therefore, $\by \in \mY \subset \bWdpoT$ holds. Moreover, $\bWopc \subset \bHo$ yields $\bHmo \subset \bWmop$ and, consequently, $\bg \in L^2(I;\bWmop)$. We also have that $\p \in W^{-1,2}(I;L^p(\Omega)/\mathbb{R})$. Now, from Theorem \ref{T2.2} we infer that $(\by,\p)$ is the unique solution of \eqref{E2.14} in $\bWdpoT \times W^{-1,2}(I;L^p(\Omega)/\mathbb{R})$.

On the other hand, \eqref{E2.14} can be considered as a Stokes system with the right hand side belonging to $L^2(I;\bHmo)$. Hence, it is well known that there exists a unique element $(\hat\by,\hat\p) \in \bWoT \times W^{-1,\infty}(I;L^2(\Omega)/\mathbb{R})$ solution of \eqref{E2.14}. Using again that $\bWoT \times W^{-1,\infty}(I;L^2(\Omega)/\mathbb{R}) \subset \bWdpoT \times W^{-1,2}(I;L^p(\Omega)/\mathbb{R})$, we deduce from Theorem \ref{T2.2} that $(\hat\by,\hat\p) = (\by,\p)$. Thus, we have $\by \in \bWoT$. Therefore, we can multiply the equation \eqref{E2.14} by $\by$ and after integration by parts it yields
\begin{align*}
&\frac{1}{2}\frac{d}{dt}\|\by(t)\|^2_\bLd + \nu\|\by(t)\|^2_\bHo = -b(\by_1,\by,\by) - b(\by,\by_2,\by)\\
& = - b(\by,\by_2,\by) \le \|\by_2\|_\bHo\|\by\|^2_\bLf \le C\|\by_2\|_\bHo\|\by\|_\bLd\|\by\|_\bHo\\
& \le \frac{\nu}{2}\|\by\|^2_\bHo + \frac{C^2}{2\nu}\|\by_2\|^2_\bHo\|\by\|^2_\bLd.
\end{align*}
From this inequality we deduce that
\[
\frac{d}{dt}\|\by(t)\|^2_\bLd \le \frac{C^2}{\nu}\|\by_2\|^2_\bHo\|\by\|^2_\bLd.
\]
Since $\by(0) = 0$, we infer from Gronwall's inequality that $\by = 0$, and with \eqref{E2.14} $\p$ is the zero element of $W^{-1,\infty}(I;L^p(\Omega)/\mathbb{R})$.
\end{proof}

\begin{rmrk}
Let us observe that $\mY \subset L^4(I;\bLf)$. Indeed, given $\by \in \mY$, we can write it in the form $\by = \by_N + \by_S$ with $\by_N \in \bWoT$ and $\by_S \in \bWqpoT$. Using a Gagliardo inequality we obtain for almost every $t \in (0,T)$
\[
\|\by_N(t)\|^4_\bLf \le C_4\|\by_N(t)\|^2_\bLd\|\by_N(t)\|^2_\bHo \le C_4\|\by_N\|^2_{L^\infty(I;\bLd)}\|\by_N(t)\|^2_\bHo.
\]
The embeddings $\bWoT \subset L^2(I;\bHo)$ and $\bWoT \subset L^\infty(I;\bLd)$ and the above inequality imply $\by_N \in L^4(I;\bLf)$. On the other hand, since $\bWqpoT \subset L^q(I;\bWp) \subset L^q(I;\bLf) \subset L^4(I;\bLf)$,  recall \eqref{E2.2}, we infer that $\by_S \in L^4(I;\bLf)$.
\label{R2.2}
\end{rmrk}

The solution of \eqref{E1.1} enjoys a better regularity than the one established in the previous remark for $q \ge 8$ and under additional assumption on $\by_0$.

\begin{thrm}
Let us assume that $q \ge 8$ and $\by_0 = \by_{N0} + \by_{S0} \in \mathbf{B}_{2,4}(\Omega) + \bBqp$. Then the variational solution $\by$ of \eqref{E1.1} belongs to $L^q(I;\bLf)$ and depends continuously in this topology on $\bef$ and $\by_0$. Moreover, the estimate
\begin{equation}
\|\by\|_{L^q(I;\bLf)} \le \eta_q\Big(\|\bef\|_{L^q(I;\bWmop)} + \|\by_{S0}\|_\bBqp + \|\by_{N0}\|_{\mathbf{B}_{2,4}(\Omega)}\Big)
\label{Estimate}
\end{equation}
holds for an increasing monotone function $\eta_q:[0,\infty) \longrightarrow [0,\infty)$  independent of $\bef$ and $\by_0$, with $\eta_q(0) = 0$.
\label{T2.6}
\end{thrm}

\begin{proof}
As in the proof of Theorem \ref{T2.1}, we decompose the equation \eqref{E1.1} in two systems, namely \eqref{E2.12} and \eqref{E2.13}. The solution $\by_S$ of \eqref{E2.12} belongs to $\bWqpoT \subset L^q(I;\bLf)$ due to the assumption \eqref{E2.2} on $p$. We prove that $\by_N \in C([0,T];\bLf)$ if $q = 8$. To this end we follow a fixed point approach. Given $\bz \in L^8(I;\bLf)$ we consider the equation
\begin{equation}
 \left\{\begin{array}{l}\displaystyle\frac{\partial \by}{\partial t} -\nu\Delta\by + \nabla\p = \bg_{\bz}\ \text{ in } Q,\\[1.2ex]\div\by = 0 \ \text{ in } Q, \ \by = 0 \ \text{ on } \Sigma,\ \by(0) = \by_{N0} \text{ in } \Omega,\end{array}\right.
 \label{E2.15}
 \end{equation}
 where
 \begin{equation}
 \bg_{\bz} = -(\by_S\cdot\nabla)\by_S - (\bz \cdot \bna)\bz - (\by_S \cdot \bna)\bz - (\bz \cdot \bna)\by_S.
 \label{E2.16}
 \end{equation}
 It is immediate to check that
 \begin{equation}
 \|\bg_{\bz}\|_{L^4(I;\bHmo)} \le (\|\by_S\|_{L^8(I;\bLf)} + \|\bz\|_{L^8(I;\bLf)})^2.
 \label{E2.17}
 \end{equation}
Then, Theorem \ref{T2.2} implies that the solution $\by_{\bz}$ of \eqref{E2.15} belongs to $\mathbf{W}_{4,2}(0,T)$ and satisfies
\begin{equation}
\|\by_{\bz}\|_{\mathbf{W}_{4,2}(0,T)} \le C_{4,2}\Big(\|\bg_{\bz}\|_{L^4(I;\bHmo)} + \|\by_{N0}\|_{\mathbf{B}_{2,4}(\Omega)}\Big).
\label{E2.18}
\end{equation}
We apply \cite[Theorem 3]{Amann2001} with
\[
\frac{1}{8} < s < \frac{1}{4}\ \text{ and } \ \theta = \frac{3}{4}
\]
to deduce that $\mathbf{W}_{4,2}(0,T) \subset L^\rho(I;\mathbf{H}^{\frac{1}{2}}(\Omega)) \subset L^\rho(I;\bLf)$ with $\rho = \frac{4}{1 - 4s}$. The choice of $s$ implies that $\rho > 8$. Moreover, the first embedding is compact. Using \cite[Th. III-4.10.2]{Amann1995} we also have that $\mathbf{W}_{4,2}(0,T) \subset C([0,T];(\bHmo,\bHo)_{\frac{3}{4},4})$.
Note that  $(\bHmo,\bHo)_{\frac{3}{4},4}  \subset   \bLf$. Indeed, by \cite[pg.186, 317]{Triebel1978} we have
$(\mathbf{H}^{-1}(\Omega),\mathbf{H}^{1}(\Omega))_{\frac{3}{4},4}=\mathbf{B}_{2,4}^{\frac{1}{2}}(\Omega)$,
where $\mathbf{B}_{2,4}^{\frac{1}{2}}(\Omega)$ denotes a Besov space.  Further we have the embedding $\mathbf{B}_{2,4}^{\frac{1}{2}}(\Omega)\subset \mathbf{L}^4(\Omega)$
\cite[pg.328]{Triebel1978}. Since $\mathbf{H}_0^1(\Omega) \subset \mathbf{H}^1(\Omega)$, the inclusion $(\mathbf{H}^{-1}(\Omega), \mathbf{H}_0^1(\Omega))_{3/4,4} \subset (\mathbf{H}^{-1}(\Omega),\mathbf{H}^1(\Omega))_{3/4,4}$ follows. Combining these facts  we find $(\mathbf{H}^{-1}(\Omega),\mathbf{H}_0^{1}(\Omega))_{\frac{3}{4},4}\subset \mathbf{L}^4(\Omega)$, and  $\mathbf{W}_{4,2}(0,T)  \subset C([0,T];\bLf)$ follows.
%{\tcb{ Note that  $(\bHmo,\bHo)_{\frac{3}{4},4}  \subset   \bLf$. Indeed, by \cite[pg.186 and 317]{Triebel1978} we have
%$(\mathbf{H}^{-1}(\Omega),\mathbf{H}^{1}(\Omega))_{\frac{3}{4},4}=\mathbf{B}_{2,4}^{\frac{1}{2}}(\Omega)$,
%where $\mathbf{B}_{2,4}^{\frac{1}{2}}(\Omega)$ denotes a Besov space. Here we use that $\mathbf{H}^{1}(\Omega) =\mathbf{H}^1_2(\Omega)$, \cite[pg. 310 and 316]{Triebel1978}, and
%$\mathbf{H}^{-1}(\Omega) =\mathbf{H}^{-1}_2(\Omega)$, \cite[Section 4.8, page 332]{Triebel1978}.
%For the definitions of $\mathbf{H}^1_2(\Omega)$ and $\mathbf{H}^{-1}_2(\Omega)$ we refer to \cite[pg.310]{Triebel1978}. Further we have the continuous embedding $\mathbf{B}_{2,4}^{\frac{1}{2}}(\Omega)\subset \mathbf{L}^4(\Omega)$
%\cite[pg.328]{Triebel1978}. Since $\mathbf{H}_0^1(\Omega) \subset \mathbf{H}^1(\Omega)$, it follows that $(\mathbf{H}^{-1}(\Omega), \mathbf{H}_0^1(\Omega))_{3/4,4} \subset (\mathbf{H}^{-1}(\Omega),\mathbf{H}^1(\Omega))_{3/4,4}$, \cite[pg. 380]{Brenner-Scott2008}. Combining these inclusions $(\mathbf{H}^{-1}(\Omega),\mathbf{H}_0^{1}(\Omega))_{\frac{3}{4},4}\subset \mathbf{L}^4(\Omega)$ as claimed, and  $\mathbf{W}_{4,2}(0,T)  \subset C([0,T];\bLf)$ follows.}}
This embedding, \eqref{E2.18} and H\"older's inequality imply for $0 < \bar t \le T$
\begin{align*}
&\|\by_{\bz}\|_{L^8(0,\bar t;\bLf)} \le \bar t^{1/8}\|\by_{\bz}\|_{C([0,\bar t];\bLf)} \le C_1\bar t^{1/8}\Big(\|\bg_{\bz}\|_{L^4(0,\bar t;\bHmo)} + \|\by_{N0}\|_{\mathbf{B}_{2,4}(\Omega)}\Big)\\
&\le C_1\bar t^{1/8}\Big[(\|\by_S\|_{L^8(I;\bLf)} + \|\bz\|_{L^8(0,\bar t;\bLf)})^2 + \|\by_{N0}\|_{\mathbf{B}_{2,4}(\Omega)}\Big].
\end{align*}
Let us take $R > 0$ and $\bz$ any element of the closed ball $\bar B_R(0)$ of $L^8(0,\bar t;\bLf)$. Then for some $0 < \bar t \le T$ small enough we obtain from the above inequality
\[
\|\by_{\bz}\|_{L^q(0,\bar t;\bLf)} \le C_1\bar t^{1/8}\Big[(\|\by_S\|_{L^8(I;\bLf)} + R)^2 + \|\by_{N0}\|_{\mathbf{B}_{2,4}(\Omega)}\Big] \le R.
\]
Hence, we have a compact mapping $\bz \in \bar B_R(0) \to \by_{\bz} \in \bar B_R(0)$. From Schauder's fixed point theorem we infer the existence of a fixed point. Since the solution of \eqref{E1.1} is unique, this fixed point must be $\by_N$ and, consequently, $\by_N$ belongs to $L^8(0,\bar t;\bLf)$. From \eqref{E2.15} with $\bg_\bz$ replaced by $\bg_\by$ we infer that $\by \in \mathbf{W}_{4,2}(0,\bar t)$ and satisfies \eqref{E2.18}. Therefore, there exists a maximal time $T^* \le T$ and a solution in the space $\mathbf{W}_{4,2}$. We know there are two possibilities: either $T^* = T$ and the theorem is proved, or $T^* < T$ and
\[
\lim_{\bar t \to T^*}\|\by_N\|_{\mathbf{W}_{4,2}(0,\bar t)} = \infty \ \text{ and } \ \|\by_N\|_{\mathbf{W}_{4,2}(0,\bar t)} < \infty \ \forall \bar t < T^*.
\]

Let us prove that the second option can not occur. Given $\varepsilon > 0$, we know from Remark \ref{R2.2} that there exists $t_\varepsilon > 0$ close enough to $T^*$ so that
\[
\|\by_N\|_{L^4(t_\varepsilon,T^*;\bLf)} < \varepsilon.
\]
From \eqref{E2.18} and \eqref{E2.17} with $\bz$ replaced by $\by$, and continuous embedding $\mathbf{W}_{4,2}(t_\varepsilon,\bar t) \subset C([t_\varepsilon,\bar t];\bLf)$ we get for every $\bar t \in (t_\varepsilon,T^*)$ and any $\varepsilon > 0$
\begin{align*}
&\|\by_N\|_{\mathbf{W}_{4,2}(t_\varepsilon,\bar t)} \le C_{4,2}\Big((\|\by_S\|_{L^8(t_\varepsilon,\bar t;\bLf)} + \|\by_N\|_{L^8(t_\varepsilon,\bar t;\bLf)})^2 + \|\by_N(t_\varepsilon)\|_{\mathbf{B}_{2,4}(\Omega)}\Big)\\
&\le C_{4,2}\Big((\|\by_S\|_{L^8(0,T;\bLf)} + \|\by_N\|^{1/2}_{C([t_\varepsilon,\bar t];\bLf)}\|\by_N\|^{1/2}_{L^4(t_\varepsilon,\bar t;\bLf)})^2 + \|\by_N(t_\varepsilon)\|_{\mathbf{B}_{2,4}(\Omega)}\Big)\\
&\le C_{4,2}\Big((\|\by_S\|_{L^8(0,T;\bLf)} + \varepsilon^{1/2}\|\by_N\|_{C([t_\varepsilon,\bar t];\bLf)}^{1/2})^2 + \|\by_N(t_\varepsilon)\|_{\mathbf{B}_{2,4}(\Omega)}\Big)\\
&\le C_{4,2}\Big((\|\by_S\|_{L^8(0,T;\bLf)} + C_1\varepsilon^{1/2}\|\by_N\|_{\mathbf{W}_{4,2}(t_\varepsilon,\bar t)}^{1/2})^2 + \|\by_N(t_\varepsilon)\|_{\mathbf{B}_{2,4}(\Omega)}\Big).
\end{align*}
Selecting $\varepsilon = [4C_{4,2}C_1]^{-1}$, we deduce from the above inequality
\[
\|\by_N\|_{\mathbf{W}_{4,2}(t_\varepsilon,\bar t)} \le 2C_{4,2}\Big(2\|\by_S\|^2_{L^8(0,T;\bLf)} + \|\by_N(t_\varepsilon)\|_{\mathbf{B}_{2,4}(\Omega)}\Big) \quad \forall \bar t \in (t_\varepsilon,T^*),
\]
which proves that the explosion is not possible.

To prove estimate \eqref{Estimate} we proceed as above to obtain
\begin{align*}
&\|\by_N\|_{\mathbf{W}_{4,2}(0,T)}\\
&\le C_{4,2}\Big((\|\by_S\|_{L^8(0,T;\bLf)} + \|\by_N\|_{C([0,T];\bLf)}^{1/2}\|\by_N\|^{1/2}_{L^4(0,T;\bLf)})^2 + \|\by_{N0}\|_{\mathbf{B}_{2,4}(\Omega)}\Big)\\
&\le C_{4,2}\Big(2\|\by_S\|^2_{L^8(0,T;\bLf)} + 2\|\by_N\|_{C([0,T];\bLf)}\|\by_N\|_{L^4(0,T;\bLf)} + \|\by_{N0}\|_{\mathbf{B}_{2,4}(\Omega)}\Big)
\end{align*}
Applying \cite[Corollary 4]{Simon87} with $X = \mathbf{W}^{\frac{1}{8},4}(\Omega)$, $B = \bLf$, $Y = \bHmo$, and $r = 2$ we infer the compactness of the embedding $\mathbf{W}_{4,2}(0,T) \subset C([0,T],\bLf)$. Here we used that $\mathbf{W}_{4,2}(0,T) \subset  C([0,T],\bLf)$, which was established above. Then, using Lions lemma with
\[
\varepsilon = \frac{1}{4C_{4,2}\|\by_N\|_{L^4(0,T;\bLf)}}
\]
we deduce from the above inequality
\begin{align*}
&\|\by_N\|_{\mathbf{W}_{4,2}(0,T)}\\
& \le 2C_{4,2}\Big(\|\by_S\|_{L^8(0,T;\bLf)} + 2C_\varepsilon\|\by_N\|_{L^4(0,T;\bLf)}\|\by_N\|_{L^2(0,T;\bHmo)} + \|\by_{N0}\|_{\mathbf{B}_{2,4}(\Omega)}\Big).
\end{align*}
Estimate \eqref{Estimate} follows from \eqref{E2.7}, the above inequality, and the estimate obtained in Remark \ref{R2.2} for $\|\by_N\|_{L^4(0,T;\bLf)}$. Finally, the continuous dependence of $\by$ with respect to $\bef$ and $\by_0$ can be proved using Theorem \ref{T2.2} and \eqref{Estimate}.
\end{proof}
For the above proof the $L^\infty(I;\bLf)$ regularity of the solution of \eqref{E2.15} is crucial. This is obtained if $q \ge 8$.

\section{Proof of Theorem \ref{T2.2}}
\label{S3}
\setcounter{equation}{0}
We separate the proof in several steps.

{\em (i) Notation and preliminaries.} Let $\bLss$  denote the closure   in $\bLs$ of  the space $\{\bphi\in \mathbf{C}^\infty_0(\Omega):\text{div}\,  \bphi =0 \}$. Given $\bef\in \bLs$ we define the Helmholtz projection $P_s: \bLs \to \bLs$ by $P_s \bef =\bef -\nabla H$, where $ \Delta H = \text{div } \bef$ in $\Omega$, and $\bf{n}\cdot(\nabla H-\bef) = 0$ with $\bf{n}$ equal to the unit outer normal vector to $\Gamma$. We have that $\text{range} (P_s)= \bLss$, $P^2_s=P_s$, and $P'_s=P_{s'}$, for the dual operator. The Stokes operator in $\bLss$ is defined by
\begin{equation}
A_s = -\nu P_s \Delta: \mD(A_s)=  \bWs\cap\mathbf{W}^{2,s}(\Omega)\to \bLss.
\label{E3.1}
\end{equation}
It is a closed bijective operator when considered on the dense domain $\mD(A_s)\subset  \bLss$.  This operator enjoys maximal parabolic regularity, \cite{GHHSS2010}, \cite{Giga-Sohr81}, \cite[page 147]{HS2018}. More precisely, for every $(\tilde\bg,\tilde\by_0) \in {L}^r(I;\bLss) \times (\bLss,\bWs\cap\mathbf{W}^{2,s}(\Omega))_{1 - \frac{1}{r},r}$ the equation
\begin{equation}
\left\{\begin{array}{l}\displaystyle\frac{\partial \tilde \by}{\partial t}(t) + A_s \tilde \by(t) = \tilde \bg(t) \text{ for a.a. } t \in (0,T), \\[1.2ex]
\by(0) = \tilde\by_0,
\end{array}\right.
\label{E3.2}
\end{equation}
has a unique solution $\tilde\by \in L^r(I;\mathbf{W}^{2,s}(\Omega)\cap\bWs) \cap W^{1,r}(I;\bLss)$. Moreover, the inequality
\begin{equation}
\|\frac{\partial \tilde \by}{\partial t}\|_{L^r(I;\bLs)}+  \|\tilde \by\|_{L^r(I;\mathbf{W}^{2,s}(\Omega))}    \le C\Big(\|\tilde \bg\|_{L^r(I;\bLs)} + \|\tilde\by_0\|\Big)
\label{E3.3}
\end{equation}
holds for some $C$ independent of $(\tilde\bg,\tilde\by_0)$. Above the norm of $\tilde\by_0$ is taken in the interpolation space $(\bLss,\bWs\cap\mathbf{W}^{2,s}(\Omega))_{1 - \frac{1}{r},r}$.

The fractional power  ${A_s^{\frac{1}{2}}} : \mD({A_s^{\frac{1}{2}}})\subset \bLss \to \bLss$, is well-defined with $\mD({A_s^{\frac{1}{2}}})$ dense in $\bLss$ and the identity $A_s = A_s^{\frac{1}{2}}A_s^{\frac{1}{2}}$ holds. $A_s^{\frac{1}{2}}$ is an isomorphism when $\mD(\Asq)$ is endowed with the graph norm of $\Asq$. The graph norm $\|\Asq \by\|_{\bLs}$ is equivalent to the norm $\|\by\|_{\bWos}$ on $\mD(\Asq)$. Moreover, the norms $\| A_s \by \|_{\bLs}$ and $\|\by\|_{\mathbf{W}^{2,s}(\Omega)}$ are equivalent on $\mD(A_s)$, see e.g. \cite{FGS2006}. We shall use in an essential manner that
\begin{equation}
\mD(\Asq)=\mD((-\Delta)^{\frac{1}{2}})\cap \bLss =\bWs,
\label{E3.4}
\end{equation}
and $-\Delta$ is understood with homogenous boundary conditions. This was verified in \cite{Giga1985}, for domains with a
'smooth' boundary.
Using the classical result in \cite{Fujiwara1968} on the characterisation of
$\mD((-\Delta^{\frac{1}{2}}))$, we obtain that $\mD(\Asq)$ is isomorphic to $\bWs$. The case of a $C^3$ boundary can be argued as follows. First we use that $[\bLss, \mD(A_s)]_{\frac{1}{2}} =   [\bLss,  \bWs\cap\mathbf{W}^{2,s}(\Omega)]_{\frac{1}{2}}  = \mD(\Asq) $, where $[\cdot,\cdot]_{\frac{1}{2}}$ denotes complex interpolation. The second equality  follows from the fact that the Stokes operator on $\bLss$ admits an $H^\infty$ calculus \cite{NS2003}, see also \cite[pg. 149]{HS2018}. Next we note that the Helmhotz projection satisfies $P_s= P_s P_s$ and that the range of $P_s$ is given by
$\bLss$. Hence $\bLss$ is a complemented subspace of $\mathbf{L}^s(\Omega)$, see \cite[pg.~22]{Triebel1978}.  Using standard regularity results for the Stokes equation \cite{AG1991} it follows from the definition of $P_s$, that it is a bounded linear operator from $\bLs$ to $\bLss$ and from $\mathbf{W}^{2,s}(\Omega)$ to $\bWs\cap\mathbf{W}^{2,s}(\Omega)$. As a consequence we obtain that $[\bLss,  \bWs\cap\mathbf{W}^{2,s}(\Omega)]_{\frac{1}{2}} = \bWs$, by a general result on interpolation couples involving  subspaces, see \cite[pg.118]{Triebel1978}.

{\em(ii) Extending the operator $A_s$.} In the following we use arguments inspired by \cite[\S 6]{DER2017} and \cite[\S 11]{ABHR2015} where the case of second order elliptic  operators is considered. First, note that we
 can utilize the above arguments for $s$ replaced by its conjugate $s'$. Hence $A^{\frac{1}{2}}_{s'}: \bWsc \to \bLscs$  is a topological isomorphism, and  its adjoint $(\Asqc)':\bLss \to \bWsc'$ is an isomorphism too. We also have
\begin{equation}\label{E3.5}
(\Asqc)'\by =\Asq \by \text{ for all } \by \in \mD(\Asq)  = \bWs,
\end{equation}
where equality holds in $\bWsc'$. Indeed, for all $\by \in \bWs$ and $\bz \in \bWsc$ we get
\begin{equation*}
\begin{array}l
\langle (\Asqc)'\by, \bz\rangle_{\bWsc' ,\bWsc }
= \langle \by, \Asqc \bz\rangle _{\bLss,\bLscs}\\[1.5ex]
=\langle \Asq \by,  \bz\rangle _{\bLss,\bLscs}=\langle \Asq \by, \bz\rangle_{ \bWsc',\bWsc},
\end{array}
\end{equation*}
where the identity $\langle \by, \Asqc \bz\rangle _{\bLss,\bLscs} = \langle \Asq \by,  \bz\rangle _{\bLss,\bLscs}$ for all $\by \in \mD(\Asq) = \bWs$ and all $\bz \in \mD(\Asqc) = \bWsc$ is well known, see eg.\cite[equality (1.23)]{GSS2005}.
Thus \eqref{E3.5} holds. Since $\bWs$ is dense in $\bLss$ we obtain that $(\Asqc)'$ is the extension of $\Asq$ from $\bWs $ to $\bLss$. In a similar way we can argue that $ ((\Asqc)')^{-1}$ is an extension of $(\Asq)^{-1}$ from $\bLss$
to $\bWsc'$.

Now, we define the operator $\mA_{s} \in \mL(\bWs,\bWsc')$ by
\begin{equation*}
\langle \mA_s\by,\bz \rangle_{\bWsc', \bWsc} = \nu \int_\Omega \nabla \by:\nabla \bz \, dx.
\end{equation*}
Let us study some properties of this operator. First, we observe that from the definitions of $A_s$ and $\mA_s$ it follows that $A_s\by = \mA_s\by$ for all $\by \in \mD(A_s)$. Hence, $\mA_s$ is an extension of $A_s$.

Now, given $\by \in \mD(A_2) \cap \mD(\Asq)$ and $\bz \in \mD(A_2) \cap \mD(\Asqc)$, we have
\begin{align*}
&\langle (\Asqc)'\Asq \by, \bz\rangle_{\bWsc',\bWsc} = \langle \Asq \by,\Asqc \bz\rangle_{\bLss,\bLscs} =\langle A_2^{\frac{1}{2}}  \by, A_2^{\frac{1}{2}} \bz\rangle_{L^2_\sigma(\Omega),L^2_\sigma(\Omega)}\\
& =\langle A_2 \by, \bz\rangle_{L^2_\sigma(\Omega),L^2_\sigma(\Omega)}= \nu \int_\Omega \nabla \by:\nabla \bz \, dx = \langle \mA_s \by,
\bz \rangle_{\bWsc', \bWsc},
\end{align*}
where we have used that $A^{\frac{1}{2}}_{s_1}\by = A^{\frac{1}{2}}_{s_2}\by$ for every $\by \in \mD(A_{s_1}^{\frac{1}{2}}) \cap \mD(A_{s_2}^{\frac{1}{2}})$ and $1 < s_1, s_2 < \infty$; see \cite[Equality (1.30)]{GSS2005}. By density of $\mD(A_{2}) \cap \mD(\Asq)$ in $\mD(\Asq)$ and $\mD(A_{2}) \cap \mD(\Asqc)$ in $\mD(\Asqc)$ we have for all $\by\in \bWs$, and $\bz\in \bWsc$
\begin{equation}
\langle (\Asqc)'  \Asq \by,\bz\rangle_{\bWsc',\bWsc}=\langle \mA_s \by,
\bz \rangle_{\bWsc', \bWsc},
\label{E3.6}
\end{equation}
and consequently
$(\Asqc)' \Asq \by =  \mA_s\by$ for all  $\by \in \mD(A_s^{\frac{1}{2}}) = \bWs$. Since the operators $\Asq:\bWs \longrightarrow \bLss$ and $(\Asqc)':\bLss \longrightarrow \bWsc'$ are isomorphisms, the previous identity implies that $\mA_s:\bWs \longrightarrow \bWsc'$ is also an isomorphism. Moreover, taking into account that $\Asq\by \in \bWs$ for all $\by \in \mD(A_s)$ and $A_s = \Asq\Asq$, we have
\begin{equation}\label{E3.7}
(\Asqc)' A_s \by = \mA_s \Asq \by  \quad \forall\by \in \mD(A_s).
\end{equation}

{\em (iii) Maximal regularity on $\bWsc'$}. Given $(\bg,\by_{S0}) \in L^r(I;\bWmos) \times \bBrs$ we define $\tilde\bg(t) = ((\Asqc)')^{-1}\bg(t)$ and $\tilde\by_0 = ((\Asqc)')^{-1}\by_{S0}$, hence $\tilde\bg \in L^r(I;\bLss)$ and
\begin{align*}
&\tilde\by_0 = ((\Asqc)')^{-1}\by_{S0} \in ((\Asqc)')^{-1}(\bBrs) = ((\Asqc)')^{-1}\Big((\bWsc',\bWs)_{1 - \frac{1}{r},r}\Big)\\
& = \Big(((\Asqc)')^{-1}(\bWsc'),((\Asqc)')^{-1}(\bWs)\Big)_{1 - \frac{1}{r},r}  = (\bLss,\mathbf{W}^{2,s}(\Omega) \cap \bWs)_{1 - \frac{1}{r},r}.
\end{align*}
Observe that $\bWmos \subset \bWsc'$, therefore $\tilde\bg(t)$ is well defined. Let $\tilde\by \in L^r(I;\mathbf{W}^{2,s}(\Omega)\cap\bWs) \cap W^{1,r}(I;\bLss)$ be the unique solution  to \eqref{E3.2}. Now, we set $\by_S= (\Asqc)'\tilde\by$. Since $\tilde\by(t) \in \mD(A_s) \subset \bWs$ for almost all $t \in I$, from \eqref{E3.5} we know that $\by_S(t) = (\Asqc)'\tilde\by(t) = \Asq\tilde\by(t) \in \bWs$ for almost all $t \in I$. Hence, we have that $\by_S \in L^r(I;\bWs)$. Additionally, from $\frac{\partial\tilde\by}{\partial t} \in L^r(I;\bLss)$ we deduce that $\frac{\partial\by_S}{\partial t} \in L^r(I;\bWsc')$. Consequently, we have that $\by_s \in \bWrsoT$.

Multiplying the differential equation satisfied by $\tilde\by$ by the operator $(\Asqc)'$, using \eqref{E3.7}, and the identity $\by_S(t) = \Asq\tilde\by(t)$, we obtain
\[
\frac{\partial\by_S}{\partial t} = (\Asqc)'\frac{\partial\tilde\by}{\partial t} = -(\Asqc)'A_s\tilde\by + (\Asqc)'\tilde\bg = -\mA_s\Asq\tilde\by + \bg = -\mA_s\by_S + \bg.
\]
Moreover, $\by_S(0) = (\Asqc)'\tilde\by(0) = (\Asqc)'\tilde\by_0 = \by_{S0}$. Thus, we have
\begin{equation*}
\left\{\begin{array}{l}\displaystyle\frac{\partial \by_S}{\partial t} + \mA_s \by_S = \bg \text{ for a.a. } t \in (0,T), \\[1.2ex]
\by_S(0) = \by_{S0}.
\end{array}\right.
\end{equation*}
This is exactly \eqref{E2.6}. Finally, we infer from  \eqref{E3.3}
\begin{equation}\label{E3.8}
\|\by_S\|_\bWrsoT \le C (\|\bg\|_{L^r(I;(\bWsc'))} + \|\by_{S0}\|_\bBrs).
\end{equation}

{\em(iv) Introduction of the pressure.} Recall that $\by_S \in L^r(I;\bWs)$ and, hence, $\Delta\by_S \in L^r(I;\bWmos)$. We also note that $\by_S:[0,T] \longrightarrow \bBrs \subset \bWsc'$ is a continuous function. Thus, we may integrate the differential  equation satisfied by $\by_S$ to obtain for every $t\in I$ and all $\bpsi \in  \bWsc$
\begin{equation}
\langle\by_S(t)-\by_{S0}, \bpsi\rangle_{\bWsc',\bWsc} - \langle \nu\int^t_0  \Delta \by_S\,ds  +\int_0^t \bg \, ds\,, \,\bpsi \rangle_{\bWmos, \bWosc}=0.
\label{E3.9}
\end{equation}

Using that $\mA_s:\bWs \longrightarrow \bWsc'$ is an isomorphism, we deduce the existence of $\hat\by_{S0} \in \bWs$ such that for every $\bpsi \in  \bWsc$
\begin{align*}
&\langle\by_{S0},\bpsi\rangle_{\bWsc',\bWsc} = \langle\mA_s\hat\by_{S0},\bpsi\rangle_{\bWsc',\bWsc}\\
& = \nu\int_\Omega\nabla\hat\by_{S0} : \nabla\bpsi\, dx = -\nu\langle\Delta\hat\by_{S0},\bpsi\rangle_{\bWmos,\bWosc}.
\end{align*}
Moreover, using that $\by_S \in L^r(I;\bWos)$, \eqref{E3.9} can be written
\[
\Big\langle\by_s(t) - \nu\Delta\hat\by_{S0} - \nu\int^t_0  \Delta \by_S\,ds  - \int_0^t \bg \, ds\,, \,\bpsi \Big\rangle_{\bWmos, \bWosc}=0
\]
for all $\bpsi \in  \bWsc$ and almost all $t \in I$. Setting
\[
\mathbf{G}(t)= \by_S(t)- \nu\Delta\hat\by_{S0}  -\nu\int^t_0  (\Delta \by_S(s) + \bg(s)) \, ds,
\]
we have that $\mathbf{G}\in L^r(I ; \bWmos)$ and
\begin{equation}\label{E3.10}
\langle \mathbf{G}(t),\bpsi\rangle_{\bWmos, \bWosc} =0\quad \forall \bpsi \in \bWsc \text{ and for a.a. }t\in I
\end{equation}
holds. We next use  de Rham's theorem in $L^r(\bar I ; \bWmos)$; see \cite[Lemma IV-1.4.1]{Sohr2001}. It implies that there exists a unique element $\pi \in L^r(I;L^s(\Omega))$ with $\int_\Omega \pi(t)\, dx=0$ for almost all $t \in I$ such that $\mathbf{G} = -\nabla\pi$ in $I \times \Omega$.  Moreover, there exists a constant $C > 0$ such that
\begin{equation*}
\|\pi\|_{L^r(I;\bLs)} \le C\|G\|_{L^r(I;\bWmos)}
\end{equation*}
holds. Let us now set $\p_S= \frac{\partial \pi}{\partial t}$. Then $\p_S \in W^{-1,r}(I;L^s(\Omega)/\mathbb{R})$ and $(\by_S, \p_S)$ is the desired solution to \eqref{E2.5}.

{\em(v) Uniqueness.} Let us assume that $(\by_1,\p_1)$ and $(\by_2,\p_2)$ are two solution of \eqref{E2.5}. We set $\by = \by_2 - \by_1$ and $\p = \p_2 - \p_1$. Since $\by_1$ and $\by_2$ satisfy \eqref{E2.6}, we infer
\begin{equation}
\left\{\begin{array}{l}\displaystyle\frac{\partial \by}{\partial t} + \mA \by = 0 \text{ for a.a. } t \in (0,T), \\[1.2ex]
\by(0) = 0.
\end{array}\right.
\label{E3.11}
\end{equation}
Now we define $\tilde\by = [(\Asqc)']^{-1}\by = (\Asq)^{-1}\by \in D(A_s) = \mathbf{W}^{2,s}(\Omega) \cap \bWs$. Then, \eqref{E3.11} and \eqref{E3.7} yield
\[
\frac{\partial\tilde\by}{\partial t} = -[(\Asqc)']^{-1}\mA_s\by = -[(\Asqc)']^{-1}\mA_s\Asq\tilde\by = -[(\Asqc)']^{-1}(\Asqc)'A_s\tilde\by = -A_s\tilde\by.
\]
Therefore, $\tilde y$ is the solution of \eqref{E3.2} with $\tilde\bg = 0$ and $\tilde\by_0 = 0$. Hence, estimate \eqref{E3.3} implies that $\tilde\by = 0$ and, consequently, $\by = 0$ as well. Finally, subtracting the partial differential equations satisfied by $(\by_2,\p_2)$ and $(\by_1,\p_1)$, we infer that $\nabla\p = 0$. This implies that $\p$ is the zero element of $W^{-1,r}(I;L^s(\Omega)/\mathbb{R})$.

\begin{rmrk}\label{R3.1}
Recently interesting work has been carried out on the treatment of the semigroup
associated to   the Stokes equation on $\bLss$  under the assumption
that the boundary $\Gamma$ is only Lipschitz continuous. This  requires a
restriction on the range of the parameter $s$. Let us summarize how
these results can be utilized for our treatment of the Stokes equation
in $\bWmos$ if the assumption on the regularity of $\Gamma$ is relaxed
to that of  Lipschitz continuity.

Following \cite{Shen2012} we define the Stokes operator in $\bLss$ is
defined by
\begin{equation*}
A_s \by = -\nu \Delta \by + \nabla\p
\end{equation*}
with domain
\begin{equation*}
\mD(A_s)= \{\by \in \bWos:  \text{ div}\,  \by =0,  \; -\nu \Delta \by +
\nabla\p\in \bLss \text{ for some } \p \in L^s(\Omega) \}.
\end{equation*}
\end{rmrk}
We have decided to use the same notation as for the definition of $A_s$ given in \eqref{E3.1}. But this should
not be problematic since it is only used within this remark. It was
proved in \cite{Shen2012} that there exists an $\epsilon >0$ such that
for all $s$ satisfying
\begin{equation}
|\frac{1}{s} - \frac{1}{2}| < \frac{1}{2d} + \epsilon,
\label{E3.12}
\end{equation}
and all $d \ge 3$, the operator $A_s$ is sectorial with angle $0$ and with
$0$ in the resolvent set, and that $-A_s$ generates a bounded analytic
semigroup on $\bLss$. As a consequence it is densely defined and closed.
Moreover $\bWs\cap\mathbf{W}^{2,s}(\Omega)\subset \mD(A_s)$ and $A_s =
-\nu P_s \Delta$ on  $\bWs\cap\mathbf{W}^{2,s}(\Omega)$. In \cite
{Kunst2017} it was further verified that $\epsilon>0$ could further be
chosen such that $A_s$ has maximal  $\bLss$ regularity and that it
admits a bounded $H^\infty-$ calculus for all $s$ satisfying
\eqref{E3.12}. This implies that the fractional powers of  $A_s$ are
well-defined and that $[\bLss, \mD(A_s)]_{\frac{1}{2}}= \mD(\Asq)$
as in the first equality of  \eqref{E3.4} above.

The second equality in  \eqref{E3.4}, namely $\mD(\Asq)= \bWs$  was established in \cite{Tolks17} in the case of Lipschitz domains, see also \cite{Tolks18}.

Summarizing, there exists $\epsilon >0$ such that for all $s$ satisfying
\eqref{E3.12}, and all $d\ge 3$,
  the structural properties of Step (i) of the proof of Theorem
\ref{T2.2} hold. The assertion of Theorem  \ref{T2.2} can therefore be
obtained as before.  It appears to be the case that restriction $d\ge 3$
dated back to the work in \cite{Shen2012}, where the proof utilizes  the
fundamental system of the Stokes system.

\section{Proof of Proposition \ref{T2.3}}
\label{S4}
\setcounter{equation}{0}
\begin{proof}[Proof of Proposition \ref{T2.3}] Let us consider the classical operator associated with the Stokes system $A:\bV \longrightarrow \bV'$ given by $\langle A\bpsi,\bphi\rangle_{\bV',\bV} = a(\bpsi,\bphi)$ $\forall \bpsi, \bphi \in \bV$. As usual, we take a base $\{\bpsi_j\}_{j = 1}^\infty$ of $\bV$ formed by eigenfunctions of $A$: $A\bpsi_j = \lambda_j\bpsi_j$ with $\{\lambda_j\}_{j = 1}^\infty \subset (0,\infty)$, $j \ge 1$. We assume that $\{\bpsi_j\}_{j = 1}^\infty$ is orthonormal for the Hilbert product in $\bH$: $(\bpsi_i,\bpsi_j)_\bLd = \delta_{ij}$. Let us denote by $\bV_k$ the subspace generated by $\{\bpsi_1,\ldots,\bpsi_k\}$. Now we define the  orthogonal $\bLd$-projection operator $P_k:\bH \longrightarrow \bV_k$ given by
\[
P_k\bpsi = \sum_{j = 1}^k(\bpsi,\bpsi_j)_\bLd\bpsi_j.
\]
Following the classical Faedo-Galerkin approach, we discretize \eqref{E2.11}
\begin{equation}
\left\{
\begin{array}{l}
\displaystyle\frac{d}{dt}(\by_k(t),\bpsi_j)_\bLd + a(\by_k(t),\bpsi_j) + \nu_0b(\by_k(t),\by_k(t),\bpsi_j) + b(\be_1(t),\by_k(t),\bpsi_j) \vspace{2mm}\\
+ b(\by_k(t),\be_2(t),\bpsi_j) = \langle\bg(t),\bpsi_j\rangle_{\bV',\bV}\ \text{ in } (0,T),\ \ 1 \le j \le k,\\
\by_k(0) = \by_{0k},
\end{array}\right.
\label{E4.1}
\end{equation}
where
\[
\by_k(t) = \sum_{j = 1}^kg_{k,j}(t)\bpsi_j\quad \text{and}\quad \by_{0k} = P_k\by_{N0} = \sum_{j = 1}^k(\by_{N0},\bpsi_j)_\bLd\bpsi_j.
\]
It is clear that \eqref{E4.1} has a maximal solution $\by_k$ defined in an interval $I_k$. We shall derive a point-wise a-priori bound for $\by_k$ on $I$ for each $k$ from which $I_k=I$ follows. We will also prove a-priori estimates which allow to pass to the limit in \eqref{E4.1}.

{\em I - Estimates in $L^\infty(I;\bLd)$.} Multiplying the equation \eqref{E4.1} by $g_{k,j}(t)$ and adding from $j = 1$ to $k$ we infer
\begin{equation}
\frac{1}{2}\frac{d}{dt}\|\by_k(t)\|_{\bLd}^2 + a(\by_k(t),\by_k(t)) + b(\by_k(t),\be_2(t),\by_k(t)) = \langle\bg(t),\by_k(t)\rangle_{\bV',\bV},
\label{E4.2}
\end{equation}
where we have used the identities $b(\by_k(t),\by_k(t),\by_k(t)) = b(\be_1(t),\by_k(t),\by_k(t)) = 0$. Let us estimate the term $b(\by_k(t),\be_2(t),\by_k(t))$. To this end we write $\be_2 = \be_{2H} + \be_{2W}$ with $\be_{2H} \in L^2(I;\bV) \cap L^\infty(I;\bLd)$ and $\be_{2W} \in L^q(I;\bWp)$. Then, we have
\[
b(\by_k(t),\be_2(t),\by_k(t)) = b(\by_k(t),\be_{2H}(t),\by_k(t)) + b(\by_k(t),\be_{2W}(t),\by_k(t)).
\]
To estimate the first term we use Schwarz's inequality, Gagliardo inequality \eqref{E7.1} with $r = 4$, and Young's inequality as follows
\begin{align}
&|b(\by_k(t),\be_{2H}(t),\by_k(t))| \le \|\by_k(t)\|^2_{\bLf}\|\be_{2H}(t)\|_{\bHo}\notag\\
&\le C_4\|\by_k(t)\|_{\bLd}\|\by_k(t)\|_\bHo\|\be_{2H}(t)\|_{\bHo}\notag\\
&\le \frac{\nu}{8}\|\by_k(t)\|_\bHo^2 + \frac{2C_4^2}{\nu}\|\by_k(t)\|^2_\bLd\|\be_{2H}(t)\|^2_{\bHo}. \label{E4.3}
\end{align}
For the second term, we apply H\"older's inequality, Gagliardo inequality with $r = 2p'$, and Young's inequality to get
\begin{align}
&|b(\by_k(t),\be_{2W}(t),\by_k(t))| \le \|\by_k(t)\|^2_{\mathbf{L}^{2p'}(\Omega)}\|\be_{2H}(t)\|_{\bWp}\notag\\
&\le C_{2p'}\|\by_k(t)\|^{\frac{2}{p'}}_{\bLd}\|\by_k(t)\|^{\frac{2}{p}}_\bHo\|\be_{2H}(t)\|_{\bWp}\notag\\
&\le \frac{\nu}{8}\|\by_k(t)\|_\bHo^2 + \frac{2C^2_{2p'}}{\nu}\|\by_k(t)\|^2_\bLd\|\be_{2W}(t)\|^{p'}_{\bWp}. \label{E4.4}
\end{align}
Since $p \ge \frac{4}{3}$, then $p' \le 4 < q$ holds. Therefore, the function $t \to \|\be_{2W}(t)\|^{p'}_{\bWp}$ is integrable in $[0,T]$. From \eqref{E4.3} and \eqref{E4.4} we infer
\[
|b(\by_k(t),\be_2(t),\by_k(t))| \le \frac{\nu}{4}\|\by_k(t)\|_\bHo^2 + C\|\by_k(t)\|^2_\bLd\Big(\|\be_{2H}(t)\|^2_{\bHo} + \|\be_{2W}(t)\|^{p'}_{\bWp}\Big).
\]
Now, inserting this inequality in \eqref{E4.2} we deduce
\begin{align*}
&\frac{1}{2}\frac{d}{dt}\|\by_k(t)\|_{\bLd}^2 + a(\by_k(t),\by_k(t))  \le \frac{\nu}{2}\|\by_k\|^2_\bHo + \frac{1}{\nu}\|\bg(t)\|^2_{\bV'}\\
& + C\|\by_k(t)\|^2_\bLd\Big(\|\be_{2H}(t)\|^2_{\bHo} + \|\be_{2W}(t)\|^{p'}_{\bWp}\Big),
\end{align*}
consequently
\begin{align}
&\frac{d}{dt}\|\by_k(t)\|_{\bLd}^2 + \nu\|\by_k(t)\|^2_\bHo \notag\\
&\le \frac{2}{\nu}\|\bg(t)\|^2_{\bV'} + 2C\|\by_k(t)\|^2_\bLd\Big(\|\be_{2H}(t)\|^2_{\bHo} + \|\be_{2W}(t)\|^{p'}_{\bWp}\Big). \label{E4.5}
\end{align}
Gronwall's inequality implies $\forall t \in I$
\begin{align*}
&\|\by_k(t)\|_{\bLd}\\
& \le \Big(\|\by_{0k}\|_\bLd + \sqrt{\frac{2}{\nu}}\|\bg\|_{L^2(I;\bV')}\Big)\exp\Big\{C\big(\|\be_{2H}\|^2_{L^2(I;\bHo)} + \|\be_{2W}\|_{L^{p'}(I;\bWp)}^{p'}\big)\Big\}\\
&\le \hat\eta_N\Big(\|\be_{2H}\|_{L^2(I;\bHo)} + \|\be_{2W}\|_{L^{p'}(I;\bWp)}\Big)\Big(\|\by_{0k}\|_\bLd + \|\bg\|_{L^2(I;\bV')}\Big)\\
\end{align*}
with
\[
\hat{\eta}_N(\rho) = \max\Big\{1,\sqrt{\frac{2}{\nu}}\Big\}\exp\{C(1 + \rho^{p'})\}.
\]
Moreover, by taking the infimum among all elements $\be_{2H} \in L^2(I;\bV) \cap L^\infty(I;\bLd)$ and $\be_{2W} \in L^q(I;\bWp)$ satisfying $\be_2 = \be_{2H} + \be_{2W}$, and noting that $\|\by_{0k}\|_\bLd = \|P_k\by_{N0}\|_\bLd \le \|\by_{N0}\|_\bLd$ we conclude
\begin{equation}
\|\by_k\|_{L^\infty(I;\bLd)} \le \hat\eta_N\big(\|\be_2\|_Y\big)\Big(\|\by_{N0}\|_\bLd + \|\bg\|_{L^2(I;\bV')}\Big).
\label{E4.6}
\end{equation}
This estimate implies that $I_k = [0,T]$ for all $k$.

{\em II - Estimates in $L^2(I;\bV)$.}  Integrating inequality \eqref{E4.5} in $[0,T]$ and using \eqref{E4.6} we get
\begin{align*}
&\nu\int_0^T\|\by_k(t)\|^2_\bHo\, dt \le \|\by_{N0}\|^2_\bLd + \frac{2}{\nu}\|\bg\|^2_{L^2(I;\bV')}\\
& + 2 C \hat\eta\big(\|\be_2\|_\bY\big)^2\big[\|\be_{2H}\|^2_{L^2(I;\bHo)} + \|\be_{2W}\|^{p'}_{L^{p'}(I;\bWp)}\big]\Big(\|\by_{N0}\|^2_\bLd + \|\bg\|^2_{L^2(I;\bV')}\Big).
\end{align*}
Arguing as above, we infer from this inequality
\begin{align}
&\|\by_k\|_{L^2(I;\bHo)}\notag\\
& \le \frac{1}{\sqrt{\nu}}\max\Big\{1,\sqrt{\frac{2}{\nu}}\Big\}\Big[1 + \sqrt{2C}\hat\eta\big(\|\be_2\|_\bY\big)\|\be_2\|^{\frac{p'}{2}}_\bY\big]\big(\|\by_{N0}\|_\bLd + \|\bg\|_{L^2(I;\bV')}\big).
\label{E4.7}
\end{align}
Then, \eqref{E4.6} and \eqref{E4.7} imply that $\by_k$ satisfies the first inequality of \eqref{E2.10} for
\[
\eta_N(\rho) = \hat\eta(\rho) + \frac{1}{\sqrt{\nu}}\max\Big\{1,\sqrt{\frac{2}{\nu}}\Big\}\big[1 + \sqrt{2 C}\hat\eta(\rho)\rho^{\frac{p'}{2}}\big].
\]

{\em III - Estimates of $\by_k'$ in $L^2(I;\bV')$.} Let us observe that $\{\bpsi_j\}_{j = 1}^\infty$ is an orthonormal basis of $\bH$ and an orthogonal basis of $\bV$. Then, given $\bpsi \in \bV$ we have the identity
\[
\bpsi = \sum_{j = 1}^\infty\frac{(\bpsi,\bpsi_j)_\bHo}{(\bpsi_j,\bpsi_j)_\bHo}\bpsi_j.
\]
From here we infer
\[
(P_k\bpsi,\bphi)_\bHo = (\bpsi,\bphi)_\bHo \quad \forall \bphi \in \bV_k,
\]
which shows that $P_k\bpsi$ is also the orthogonal $\bHo$-projection of $\bpsi$ on $\bV_k$. As a consequence we get that $\|P_k\bpsi\|_\bHo \le \|\bpsi\|_\bHo$ for every $\bpsi \in \bV$. Moreover, since
\[
\by_k'(t) = \sum_{j = 1}^kg'_{k,j}(t)\bpsi_j,
\]
we get $(\by_k'(t),\bpsi)_\bLd = (\by_k'(t),P_k\bpsi)_\bLd$. Then, from the differential equation \eqref{E2.11}
we get for every $\bpsi \in V$
\begin{align*}
&(\by_k'(t),\bpsi)_\bLd = \langle\bg(t),P_k\bpsi\rangle_{\bV',\bV} - a(\by_k(t),P_k\bpsi) - \nu_0b(\by_k(t),\by_k(t),P_k\bpsi)\\
&- b(\be_1(t),\by_k(t),P_k\bpsi) - b(\by_k(t),\be_2(t),P_k\bpsi).
\end{align*}
Now we observe that the estimates I and II imply that $\{\by_k\}_{k = 1}^\infty$ is bounded in $L^2(I;\bV) \cap L^\infty(I;\bH)$, hence bounded in $\bY$. Then, using Lemma \ref{L2.1} we deduce form the above identity and the inequality $\|P_k\bpsi\|_\bHo \le \|\bpsi\|_\bHo$
\[
\|\by_k'\|_{L^2(I;\bV')} \le \|\bg\|_{L^2(I;\bV')} +\|\by_k\|_\bY\big( \nu + \nu_0\|\by_k\|_\bY + \|\be_1\|_\bY + \|\be_2\|_\bY\big).
\]
This inequality along with the estimates \eqref{E4.6} and \eqref{E4.7} proves the boundedness of $\{\by_k'\}_{k = 1}^\infty$ in $L^2(I;\bV')$. Hence, $\{\by_k\}_{k = 1}^\infty$ is bounded in $\bWoT$ and the second inequality of \eqref{E2.10} holds.

Finally, using the above estimates, it is standard to pass to the limit in \eqref{E4.1}, taking a subsequence if necessary, and to deduce that $\{\by_k\}_{k = 1}^\infty$ converges weakly in $\bWoT$ and weakly$^*$ in $L^\infty(I;\bH)$ to a solution $\by_N$ of the system \eqref{E2.11}; see, for instance, \cite[Chapter I-6.4.4]{Lions69}. Moreover, since every $\by_k$ satisfies \eqref{E2.10}, then $\by_N$ does it as well. Further, since $\by_N \in L^2(I;\bV)$ and $\by'_N \in L^2(I;\bV')$ we deduce that $\by_N \in \bWoT$. Moreover, applying De Rham theorem we infer the existence of $\p_N \in W^{-1,\infty}(I;L^2(\Omega)/\mathbb{R})$ such that $(\by_N,\p_N)$ is solution of \eqref{E2.9}; see \cite[Chaper V-1.5]{Boyer-Fabrie2013}. Now, using Lemma \ref{L2.1} and Gronwall's inequality the uniqueness of a solution follows in a standard way; cf. \cite[Chaper V-1.3.5]{Boyer-Fabrie2013}.
\end{proof}

\section{Sensitivity analysis of the state equation}
\label{S5}
\setcounter{equation}{0}

In this section we analyze the differentiability of the mapping $G:L^q(I;\bWmop) \longrightarrow \mY$ associating to each element $\bef \in L^q(I;\bWmop)$ the solution $\by_\bef \in \mY$ of \eqref{E2.3}.

\begin{thrm}
The mapping $G$ is of class $C^\infty$. Further, given $\bef, \bg, \bg_1, \bg_2 \in L^q(I;\bWmop)$ we have that $\bz_\bg = G'(\bef)\bg$ and $\bz_{\bg1,\bg2} = G''(\bef)(\bg_1,\bg_2)$ are the unique solutions of the systems
\begin{equation}
\left\{\begin{array}{l}\displaystyle\frac{\partial \bz}{\partial t} -\nu\Delta\bz + (\by_\bef \cdot \bna)\bz + (\bz \cdot \bna)\by_\bef + \nabla\q = \bg\ \text{ in } Q,\\[1.2ex]\div\bz = 0 \ \text{ in } Q, \ \bz = 0 \ \text{ on } \Sigma,\ \bz(0) = 0 \text{ in } \Omega,\end{array}\right.
\label{E5.1}
\end{equation}
and
\begin{equation}
\left\{\begin{array}{l}\displaystyle\frac{\partial \bz}{\partial t} -\nu\Delta\bz + (\by_\bef \cdot \bna)\bz + (\bz \cdot \bna)\by_\bef + \nabla\q = - (\bz_{\bg_2} \cdot \bna)\bz_{\bg_1} - (\bz_{\bg_1} \cdot \bna)\bz_{\bg_2}\ \text{ in } Q,\\[1.2ex]\div\bz = 0 \ \text{ in } Q, \ \bz = 0 \ \text{ on } \Sigma,\ \bz(0) = 0 \text{ in } \Omega,\end{array}\right.
\label{E5.2}
\end{equation}
respectively, where $\by_\bef = G(\bef)$ and $\bz_{\bg_i} = G'(\bef)\bg_i$ for $ i = 1, 2$.
\label{T3.1}
\end{thrm}

\begin{proof}
Let us define the space $\mF = L^2(I;\bV') + L^q(I;\bWpc')$ endowed with the norm
\[
\|\bef\|_\mF = \inf\{\|\bef_1\|_{L^2(I;\bV')} + \|\bef_2\|_{L^q(I;\bWpc')} : \bef = \bef_1 + \bef_2\}.
\]
Thus, $\mF$ is a Banach space. We also consider the operators
\begin{align*}
&A_\bV:\bV \longrightarrow \bV',\ \ \langle A_\bV\by,\bz\rangle_{\bV',\bV} = \nu\int_\Omega \nabla y : \nabla z\, dx,\ \ \forall \bz \in\bV,\\
&A_\bW:\bWp \longrightarrow \bWpc',\ \ \langle A_\bV\by,\bz\rangle_{\bWpc',\bWp} = \nu\int_\Omega \nabla y : \nabla z\, dx,\ \ \forall \bz \in \bWpc.
\end{align*}
Associated with these two continuous operators we define
\[
A:\mY \longrightarrow \mF,\ \ A\by = A_\bV\by_1 + A_\bW\by_2,
\]
where $\by = \by_1 + \by_2$ with $\by_1 \in \bWoT$ and $\by_2 \in \bWqpoT$. It is immediate to check that $A\by$ is independent of the chosen representation $\by = \by_1 + \by_2$, and it is continuous. Now, we introduce the mapping
\begin{align*}
&\mG:\mY \times L^q(I;\bWpc') \longrightarrow \mF \times \bY_0,\\
&\mG(\by,\bef) = \Big(\frac{\partial\by}{\partial t} + A\by + B(\by,\by) - \bef,\by(0) - \by_0\Big),
\label{E5.3}
\end{align*}
where $\by_0$ is the initial condition in \eqref{E2.1}. Recall that $\mY \subset C([0;T];\bY_0)$ holds. Hence, $\by \in \mY \to \by(0) \in \bY_0$ is a linear and continuous mapping. Moreover, Lemma \ref{L2.1} implies that $\by \in \mY \to B(\by,\by) \in L^2(I;\bHmo) \subset L^2(I;\bV') \subset \mF$ is bilinear and continuous. By definition of $\bWoT$ and $\bWqpoT$ we also have that $\frac{\partial}{\partial t}:\mY \to \mF$ is a linear and continuous operator. All together this implies that $\mG$ is a $C^\infty$ mapping.

Given $\bef \in L^q(I;\bWpc')$, we denote by $\by_\bef \in \mY$ the solution of \eqref{E2.3}. Then, we have that
\begin{align}
&\frac{\partial\mG}{\partial\by}(\by_f,\bef):\mY\longrightarrow\mF \times \bY_0,\notag\\
&\frac{\partial\mG}{\partial\by}(\by_f,\bef)\bz = \Big(\frac{\partial\bz}{\partial t} + A\bz + B(\by_f,\bz) + B(\bz,\by_f),\bz(0)\Big)\quad \forall \bz \in \mY
\end{align}
is a linear and continuous mapping. Actually, it is an isomorphism. Let us prove this. Given an arbitrary element $(\bg,\bz_0) \in \mF \times \bY_0$, we set $\bg = \bg_N + \bg_S$ and $\bz_0 = \bz_{N0} + \bz_{S0}$ with $\bg_N \in L^2(I;\bV')$, $\bg_S \in L^q(I;\bWpc')$, $\bz_{N0} \in \bH$, and $\bz_{S0} \in \bBqp$. Now, we show the existence and uniqueness of a solution $\bz \in \mY$ of the equation

\begin{equation}
\left\{\begin{array}{l}\displaystyle\frac{\partial \bz}{\partial t} + A\bz + B(\by_f,\bz) + B(\bz,\by_f) = \bg \text{ in } (0,T), \\[1.2ex]
\bz(0) = \bz_0.
\end{array}\right.
\label{E5.4}
\end{equation}
We decompose the system in two parts
\begin{equation}
\left\{\begin{array}{l}\displaystyle\frac{\partial \bz_S}{\partial t} +A_\bW\bz_S = \bg_S\text{ in } (0,T), \\[1.2ex]\bz_S(0) = \bz_{S0},\end{array}\right.
\label{E5.5}
\end{equation}
and
\begin{equation}
\hspace{-0.45cm}\left\{\begin{array}{l}\displaystyle\frac{\partial \bz_N}{\partial t} +A_\bV\bz_N + B(\by_\bef,\bz_N) + B(\bz_N,\by_\bef) = \bg_N - B(\by_\bef,\bz_S) - B(\bz_S,\by_\bef)\text{ in } (0,T), \\[1.2ex] \bz_N(0) = \bz_{N0} \text{ in } \Omega.\end{array}\right.
\label{E5.6}
\end{equation}
The existence and uniqueness of a solution $\bz_S \in \bWqpoT$ of \eqref{E5.5} follows from Theorem \ref{T2.2}. In equation \eqref{E5.6}, we have that $\bz_{N0} \in \bH$, $\by_\bef \in \bY$, and from Lemma \ref{L2.1} we get that the right hand side of the partial differential equation belongs to $L^2(I;\bHmo)$. Hence, from Proposition \ref{T2.3} we infer the existence and uniqueness of a solution $\by_N \in \bWoT$ of \eqref{E5.6}. Now, setting $\by = \by_N + \by_S \in \mY$, we deduce that $\by$ is a solution of \eqref{E5.4}. The uniqueness follows from Gronwall's inequality, arguing as in the proof of Theorem \ref{T2.1}.

Now, it is enough to apply the implicit function theorem to deduce the existence of a function $\tilde G:\bWpc' \longrightarrow \mY$ of class $C^\infty$ such that $\mG(\tilde G(\bef),\bef) = 0$ for every $\bef \in L^q(I;\bWpc')$. Hence, $\tilde G(\bef) = \by_\bef$ is the solution of \eqref{E2.3}. Moreover, by differentiation with respect to $\bef$ of the identity $\mG(\tilde G(\bef),\bef) = 0$, setting $\bz_g = DG(\bef)\bg$ for $\bg \in  L^q(I;\bWmop)$, and using \eqref{E5.3} and Rham's theorem equation \eqref{E5.1} follows. The equation \eqref{E5.2} follows easily from the identity
\[
\frac{\partial^2\mG}{\partial\by^2}(\by_\bef)(\bg_1,\bg_2)\bz = \Big(\frac{\partial\bz}{\partial t} + A\bz + B(\by_f,\bz) + B(\bz,\by_\bef) + B(\bz_{\bg_1},\bz_{\bg_2}) + B(\bz_{\bg_2},\bz_{\bg_1}),\bz(0)\Big).
\]
Finally, the theorem follows by observing that $G:L^q(I;\bWmop) \longrightarrow \mY$ is given by $G = \tilde G\circ R_\sigma$, where $R_\sigma:L^q(I;\bWmop) \longrightarrow \bef \in L^q(I;\bWpc')$ is the restriction operator, that is linear and continuous.
\end{proof}

\section{Asymptotic Stability of Steady Solutions}
\label{S6}
\setcounter{equation}{0}

In this section we extend an asymptotic stability result to the case of a source term $\bef \in \bWp$, with $p\in [\frac{4}{3},2)$, independent of time, compare e.g. \cite[Section 3.4]{Boyer-Fabrie2013}. Associated with this source we consider an element $(\by_\infty,\p_\infty) \in \bWop \times L^p(\Omega)/\mathbb{R}$ satisfying the following steady-state Navier-Stokes equations
\begin{equation}
\left\{\begin{array}{l}\displaystyle -\nu\Delta\by_\infty + (\by_\infty \cdot \bna)\by_\infty + \nabla\p_\infty = \bef \text{ in } \Omega,\\[1.2ex]\div\by_\infty = 0 \ \text{ in } \Omega, \ \by_\infty = 0 \ \text{ on } \Gamma.\end{array}\right.
\label{E6.1}
\end{equation}
The reader is referred to \cite{Serre1983} for the existence of a solution of \eqref{E6.1} in the mentioned space. In this section, we will prove that given $\by_0 \in \bH$ and assuming that $\bef$ is small enough, the solution of
\begin{equation}
\left\{\begin{array}{l}\displaystyle\frac{\partial \by}{\partial t} -\nu\Delta\by + (\by \cdot \bna)\by + \nabla\p = \bef\ \text{ in } \Omega \times (0,\infty),\\[1.2ex]\div\by = 0 \ \text{ in } \Omega \times (0,\infty), \ \by = 0 \ \text{ on } \Gamma \times (0,\infty),\ \by(0) = \by_0 \text{ in } \Omega.\end{array}\right.
\label{E6.2}
\end{equation}
satisfies that $\by(t) \to \by_\infty$ as $t \to \infty$. First, let us prove that the solution of \eqref{E6.1} is unique if $\bef$ is small enough.

\begin{lmm}
There exists a constant $C_\infty$ depending only on $\nu$, $p$ and $\Omega$ such that if $\|\bef\|_{\bWmop} < C_\infty$, then  \eqref{E6.1} has a unique solution $(\by_\infty,\p_\infty) \in \bWp \times L^p(\Omega)/\mathbb{R}$.
\label{L6.1}
\end{lmm}

\begin{proof}
Let $(\by_1,\p_1)$ and $(\by_2,\p_2)$ be two solutions of \eqref{E6.1} belonging to $\bWp \times L^p(\Omega)/\mathbb{R}$. We set $(\by,\p) = (\by_2 - \by_1,\p_2-\p_1)$ and
\[
\bg = -\Big[(\by\cdot\nabla)\by + (\by_2\cdot\nabla)\by + (\by\cdot\nabla)\by_2\Big].
\]
Then, $(\by,\p) \in \bWp \times L^p(I;\Omega)/\mathbb{R}$ is the unique solution of the Stokes system
\begin{equation}
\left\{\begin{array}{l}\displaystyle -\nu\Delta\by + \nabla\p = \bg \text{ in } \Omega,\\[1.2ex]\div\by = 0 \ \text{ in } \Omega, \ \by = 0 \ \text{ on } \Gamma.\end{array}\right.
\label{E6.3}
\end{equation}
It is easy to check that $\bg \in \bHmo \subset \bWmop$. Using this fact, Cattabriga's result \cite{Cattabriga61}, see also \cite[Theorem IV.6.1]{Galdi2011}, and the embedding $\bV \subset \bWp$, we infer that $\by \in \bV$. Therefore, we can multiply the equation \eqref{E6.3} by $\by$ to get
\begin{equation}
\nu\|\by\|^2_\bHo - \langle\bg,\by\rangle_{\bHmo,\bHo} = 0.
\label{E6.4}
\end{equation}
Using that $b(\by,\by,\by) = b(\by_2,\by,\by) = 0$ and $b(\by,\by_2,\by) = - b(\by,\by,\by_2)$ it yields
\begin{align}
&|\langle\bg,\by\rangle_{\bHmo,\bHo}| = |b(\by,\by,\by_2)|\notag\\
& \le \|\by\|_\bLf\|\by_2\|_\bLf\|\by\|_\bHo \le C_1\|\by_2\|_\bLf\|\by\|^2_\bHo,
\label{E6.5}
\end{align}
where $C_1$ only depends on $\Omega$. Now, from \cite[Theorem 2.3]{Casas-Kunisch2019} we infer the existence of a constant $M_p$ depending only of $p$, $\nu$ and $\Omega$ such that
\[
\|\by_2\|_\bWop \le M_p\|\bef\|_{\bWpc'}\Big(1 + \|\by_2\|_\bLf\Big).
\]
Hence, with the inequality $\|\by_2\|_\bLf \le C_2\|\by_2\|_\bWop$, recalling that $p \ge \frac{4}{3}$, and assuming that
\[
\|\bef\|_{\bWpc'} < \frac{1}{C_2M_p},
\]
we obtain
\begin{equation}
\|\by_2\|_\bLf \le \frac{C_2M_p\|\bef\|_{\bWpc'}}{1 - C_2M_p\|\bef\|_{\bWpc'}}.
\label{E6.6}
\end{equation}
From \eqref{E6.4} and \eqref{E6.5} we get that $\by = 0$ if $C_1\|\by_2\|_\bLf < \nu$. With \eqref{E6.6} this inequality holds if
\begin{equation}
\|\bef\|_{\bWpc'} < C_\infty = \frac{\nu}{(C_1 + \nu)C_2M_p}.
\label{E6.7}
\end{equation}
Finally, the uniqueness of $\p_\infty \in L^p(\Omega)/\mathbb{R}$ follows from the uniqueness of $\by_\infty$ in $\bWp$.
\end{proof}

Now, we prove the main theorem of this section.
\begin{thrm}
Let us assume that $(\bef,\by_0) \in \bWmop \times \bH$ with $\bef$ satisfying \eqref{E6.7}. Let $\by \in \mY$ be the solution of \eqref{E6.2}. Then, there exists a constant $\alpha > 0$ depending only on $\Omega$, $p$ and $\nu$ such that the following estimate holds
\begin{equation}
\|\by(t) - \by_\infty\|_\bLd \le \|\by_0 - \by_\infty\|_\bLd\text{\rm e}^{-\alpha t},\quad \forall t \ge 0.\label{E6.8}
\end{equation}
\label{T6.1}
\end{thrm}

\begin{proof}
Let us set $(\bz,\q) = (\by - \by_\infty,\p - \p_\infty)$. Hence, we have
\begin{equation}
\left\{\begin{array}{l}\displaystyle\frac{\partial \bz}{\partial t} -\nu\Delta\bz + (\bz \cdot \bna)\bz + (\by_\infty\cdot\nabla)\bz + (\bz\cdot\nabla)\by_\infty + \nabla\q = 0\ \text{ in } Q,\\[1.2ex]\div\bz = 0 \ \text{ in } Q, \ \bz = 0 \ \text{ on } \Sigma,\ \bz(0) = \by_0 - \by_\infty\text{ in } \Omega.\end{array}\right.
\label{E6.9}
\end{equation}
First, we observe that $\by_\infty \in \bWp \subset \bH$, hence $\by_0 - \by_\infty \in \bH$. Taking $\nu_0 = 1$, $\be_1 = \be_2 = \by_\infty$, $\bg = 0$ and $\by_{N0} = \by_0 - \by_\infty$ in Proposition \ref{T2.3}, we deduce with Theorem \ref{T2.2} that $(\bz,\q) \in \bWoT \times W^{-1,\infty}(I;L^2(\Omega)/\mathbb{R})$ and that it is the unique solution of \eqref{E6.9}. Then, multiplying \eqref{E6.9} by $\bz$ and using that $b(\bz,\bz,\bz) = b(\by_\infty,\bz,\bz) = 0$ and $b(\bz,\by_\infty,\bz) = - b(\bz,\bz,\by_\infty)$ we infer
\begin{equation}
\frac{1}{2}\frac{d}{dt}\|\bz(t)\|^2_\bLd + \nu\|\bz(t)\|^2_\bHo - b(\bz(t),\bz(t),\by_\infty(t)) = 0, \ \text{ for a.a. t} \ge 0.
\label{E6.10}
\end{equation}
As in \eqref{E6.5}, we have
\[
|b(\bz(t),\bz(t),\by_\infty(t))| \le C_1\|\by_\infty\|_\bLf\|\bz\|^2_\bHo.
\]
Therefore, with \eqref{E6.6} and the inequality $\|\bz\|_\bLd \le C_3\|\bz\|_\bHo$ it follows
\begin{align*}
&\nu\|\bz(t)\|^2_\bHo - b(\bz(t),\bz(t),\by_\infty(t)) \ge \Big(\nu - \frac{C_1C_2M_p\|\bef\|_{\bWpc'}}{1 - C_2M_p\|\bef\|_{\bWpc'}}\Big)\|\bz\|^2_\bHo\\
& \ge \frac{1}{C_3^2}\Big(\nu - \frac{C_1C_2M_p\|\bef\|_{\bWpc'}}{1 - C_2M_p\|\bef\|_{\bWpc'}}\Big)\|\bz\|^2_\bLd.
\end{align*}
Taking
\[
\alpha = \frac{1}{C_3^2}\Big(\nu - \frac{C_1C_2M_p\|\bef\|_{\bWpc'}}{1 - C_2M_p\|\bef\|_{\bWpc'}}\Big),
\]
we deduce from the assumption $\|\bef\|_\bWmop < C_\infty$ that $\alpha > 0$. Moreover, from \eqref{E6.10} we obtain
\[
\frac{d}{dt}\|\bz(t)\|^2_\bLd + 2\alpha\|\bz(t)\|^2_\bLd \le 0,\ \text{ for a.a. } t \ge 0.
\]
Applying Gronwall's lemma to this inequality we deduce \eqref{E6.8}.

\end{proof}

\section{Appendix}
\label{SA}

\begin{proof}[Proof of Lemma \ref{L2.1}]
Let $\by_1, \by_2 \in \bY$ and $(\by_{i,H},\by_{i,W}) \in [L^2(I;\bV) \cap L^\infty(I;\bH)] \times L^q(I;\bWp)$ be elements such that $\by_i = \by_{i,H} + \by_{i,W}$ for $i = 1, 2$. Then, we are going to prove estimates for the terms $B(\by_{1,H},\by_{2,H})$, $B(\by_{1,H},\by_{2,W})$, $B(\by_{1,W},\by_{2,H})$, and $B(\by_{1,W},\by_{2,W})$. Given $\bpsi \in \bHo$, we observe that
\[
\langle B(\by_1,\by_2),\bpsi\rangle_{\bHmo,\bHo} = \sum_{i, j = 1}^2\int_\Omega \by_{1,i}(x,t)\partial_{x_i} \by_{2,j}(x,t)\bpsi_j(x)\, dx.
\]
To deduce the estimates we will use the Gagliardo inequality
\begin{equation}
\|\by\|_{\mathbf{L}^r(\Omega)} \le C_r\|\by\|_\bLd^{\frac{2}{r}}\|\by\|_\bHo^{\frac{r-2}{r}}\ \ \forall r \in (2,\infty) \text{ and } \forall \by \in \bHo;
\label{E7.1}
\end{equation}
see \cite[page 313]{Brezis2011}. Now, we proceed in four steps.

{\em Step 1.-} Using that $\div{\by_{1,H}} = 0$, we know that
\begin{equation}
\int_\Omega[(\by_{1,H}\cdot\nabla)\by_{2,H}]\bpsi\, dx = -\int_\Omega[(\by_{1,H}\cdot\nabla)\bpsi]\by_{2,H}\, dx.
\label{E7.2}
\end{equation}
Then, from Schwarz's inequality and \eqref{E7.1} with $r = 4$ it follows
\begin{align}
&\left(\int_0^T|\langle B(\by_{1,H}(t),\by_{2,H}(t)),\bpsi\rangle|^2\, dt\right)^{\frac{1}{2}} = \left(\int_0^T|\langle B(\by_{1,H}(t),\bpsi),\by_{2,H}(t)\rangle|^2\, dt\right)^{\frac{1}{2}}\notag\\
&\le\left(\int_0^T\|\by_{1,H}(t)\|^2_\bLf\|\by_{2,H}(t)\|^2_\bLf\, dt\right)^{\frac{1}{2}}\|\bpsi\|_\bHo\notag\\
&\le C_4^2\left(\int_0^T\|\by_{1,H}(t)\|_\bLd\|\by_{1,H}(t)\|_\bHo\|\by_{2,H}(t)\|_\bLd\|\by_{2,H}(t)\|_\bHo\, dt\right)^{\frac{1}{2}}\|\bpsi\|_\bHo\notag\\
&\le C_4^2\|\by_{1,H}\|^{\frac{1}{2}}_{L^\infty(I;\bLd)}\|\by_{1,H}\|^{\frac{1}{2}}_{L^2(I;\bHo)}\|\by_{2,H}\|^{\frac{1}{2}}_{L^\infty(I;\bLd)}\|\by_{2,H}\|^{\frac{1}{2}}_{L^2(I;\bHo)}\|\bpsi\|_\bHo\notag\\
&\le\frac{C_4^2}{4}\Big(\|\by_{1,H}\|_{L^\infty(I;\bLd)} + \|\by_{1,H}\|_{L^2(I;\bHo)}\Big)\notag\\
&\quad \times \Big(\|\by_{2,H}\|_{L^\infty(I;\bLd)} + \|\by_{2,H}\|_{L^2(I;\bHo)}\Big)\|\bpsi\|_\bHo. \label{E7.3}
\end{align}

\vspace{2mm}

{\em Step 2.-} Using H\"older's inequality and \eqref{E7.1} with $r = 2p' = \frac{2p}{p-1}$ we get
\begin{align*}
&\left(\int_0^T|\langle B(\by_{1,H}(t),\by_{2,W}(t)),\bpsi\rangle|^2\, dt\right)^{\frac{1}{2}}\\
&\le \left(\int_0^T\|\by_{1,H}(t)\|^2_{\bL^{2p'}(\Omega)}\|\bpsi\|^2_{\bL^{2p'}(\Omega)}\|\by_{2,W}(t)\|^2_{\bWop}\, dt\right)^{\frac{1}{2}}\\
&\le C\left(\int_0^T\|\by_{1,H}(t)\|^{\frac{2}{p'}}_{\bLd}\|\by_{1,H}(t)\|^{\frac{2}{p}}_\bHo\|\by_{2,W}(t)\|^2_{\bWop}\, dt\right)^{\frac{1}{2}}\|\bpsi\|_\bHo.
\end{align*}
Applying H\"older's inequality with $r_1 = \frac{pq}{pq - q - 2p}$, $r_2 = p$, and $r_3 = \frac{q}{2}$ we find
\begin{align*}
&\le C\|\by_{1,H}\|^\frac{1}{p'}_{L^{\frac{2r_1}{p'}}(I;\bLd)}\|\by_{1,H}\|^{\frac{1}{p}}_{L^2(I;\bHo)}\|\by_{2,W}\|_{L^q(I;\bWop)}\|\bpsi\|_\bHo\\
&\le CT^{\frac{1}{2r_1}}\|\by_{1,H}\|^\frac{1}{p'}_{L^\infty(I;\bLd)}\|\by_{1,H}\|^{\frac{1}{p}}_{L^2(I;\bHo)}\|\by_{2,W}\|_{L^q(I;\bWop)}\|\bpsi\|_\bHo,
\end{align*}
and  with Young's inequality we infer
\begin{align}
&\le \frac{CT^{\frac{1}{2r_1}}}{p}\left(\|\by_{1,H}\|_{L^\infty(I;\bLd)} + \|\by_{1,H}\|_{L^2(I;\bHo)}\right)\|\by_{2,W}\|_{L^q(I;\bWop)}\|\bpsi\|_\bHo.
\label{E7.4}
\end{align}
Observe that $q > \frac{2p}{p-1}$ implies $pq - q - 2p > 0$, hence $r_1 > 1$ and the use of H\"older's inequality is correct.

\vspace{2mm}

{\em Step 3.-} Using again H\"older's inequality and \eqref{E7.1} with $r = 4$ we obtain
\begin{align*}
&\left(\int_0^T|\langle B(\by_{1,W}(t),\by_{2,H}(t)),\bpsi\rangle|^2\, dt\right)^{\frac{1}{2}}\\
&\le \left(\int_0^T\|\by_{1,W}(t)\|^2_\bLf\|\by_{2,H}(t)\|^2_\bLf\, dt\right)^{\frac{1}{2}}\|\bpsi\|_\bHo\\
&\le C_4^2\left(\int_0^T\|\by_{1,W}(t)\|^2_\bWop\|\by_{2,H}(t)\|_\bLd\|\by_{2,H}(t)\|_\bHo\, dt\right)^{\frac{1}{2}}\|\bpsi\|_\bHo,
\end{align*}
applying H\"older's inequality with $r_1 = \frac{q}{2}$, $r_2 = \frac{2q}{q - 4}$, and $r_3 = 2$ we infer
\begin{align}
&\le C_4^2\|\by_{1,W}\|_{L^q(I;\bWop)}\|\by_{2,H}\|^{\frac{1}{2}}_{L^{\frac{2q}{q - 4}}(I;\bLd)}\|\by_{2,H}\|^{\frac{1}{2}}_{L^2(I;\bHo)}\|\bpsi\|_\bHo\notag\\
&\le C_4^2T^{\frac{q - 4}{4q}}\|\by_{1,W}\|_{L^q(I;\bWop)}\|\by_{2,H}\|^{\frac{1}{2}}_{L^\infty(I;\bLd)}\|\by_{2,H}\|^{\frac{1}{2}}_{L^2(I;\bHo)}\|\bpsi\|_\bHo\notag\\
&\le \frac{C_4^2T^{\frac{q - 4}{4q}}}{2}\|\by_{1,W}\|_{L^q(I;\bWop)}\Big(\|\by_{2,H}\|_{L^\infty(I;\bLd)} + \|\by_{2,H}\|_{L^2(I;\bHo)}\Big)\|\bpsi\|_\bHo. \label{E7.5}
\end{align}

\vspace{2mm}

{\em Step 4.-} Using again the property \eqref{E7.2}, H\"older's inequality, the embedding $\bWop \subset \bLf$, due to $p \ge \frac{4}{3}$, and the fact that $q > 4$ we obtain
\begin{align}
&\left(\int_0^T|\langle B(\by_{1,W}(t),\by_{2,W}(t)),\bpsi\rangle|^2\, dt\right)^{\frac{1}{2}}\notag\\
&\le \left(\int_0^T\|\by_{1,W}\|^2_{\bLf}\|\by_{2,W}\|^2_{\bLf}dt\right)^{\frac{1}{2}}\|\bpsi\|_\bHo\notag\\
&\le C\|\by_{1,W}\|_{L^4(I;\bWop)}\|\by_{2,W}\|_{L^4(I;\bWop)}\|\bpsi\|_\bHo\notag\\
&\le CT^{\frac{2(q - 4)}{q}}\|\by_{1,W}\|_{L^q(I;\bWop)}\|\by_{2,W}\|_{L^q(I;\bWop)}\|\bpsi\|_\bHo. \label{E7.6}
\end{align}

Finally, adding the  estimates  \eqref{E7.3}-\eqref{E7.6} we obtain
\begin{align}
&\|B(\by_1,\by_2)\|_{L^2(I;\bHmo)} \le C\Big(\|\by_{1,H}\|_{L^\infty(I;\bLd)} + \|\by_{1,H}\|_{L^2(I;\bHo)} + \|\by_{1,W}\|_{L^q(I;\bWop)}\Big)\notag\\
&\qquad \times\ \Big(\|\by_{2,H}\|_{L^\infty(I;\bLd)} + \|\by_{2,H}\|_{L^2(I;\bHo)} + \|\by_{2,W}\|_{L^q(I;\bWop)}\Big).\label{E7.7}
\end{align}
Taking the infimum on the right hand side of the above inequality among all functions $(\by_{i,H},\by_{i,W}) \in [L^2(I;\bHo) \cap L^\infty(0,T;\bLd)] \times L^q(0,T;\bWop)$ satisfying that $\by_i = \by_{i,H} + \by_{i,W}$, $i = 1, 2$, we conclude
\[
\|B(\by_1,\by_2)\|_{L^2(I;\bHmo)} \le C'\|\by_1\|_\bY\|\by_2\|_\bY.
\]
\end{proof}

\subsection*{Acknowledgement}
The authors are indebted to J.~Rehberg for pointing out \textit{the square root technique} to prove the maximal regularity of the Stokes system, and to M.~Hieber and J.~Saal for pointing at their book-chapter \cite{HS2018}. We are also indebted to Jean-Pierre Raymond for important hints towards the proof of Theorem \ref{T2.6}.

\end{document}